\documentclass[reqno, 12pt]{amsart} 

\usepackage{amssymb}
\usepackage{amsthm}
\usepackage{amsfonts}
\usepackage{amsmath}
\usepackage{hyperref}
\hypersetup{
   colorlinks=true,
   linkcolor=blue,
}
\usepackage{enumerate}
 \usepackage{mathtools}
\usepackage{tikz}
\usepackage{mathdots}
 \usepackage{yhmath}
\usepackage{cancel}
\usepackage{ulem,xpatch}
\usepackage{color}
\usepackage{bbm}
\usepackage{dsfont}

\newtheorem{lemma}{Lemma}[section]
\newtheorem{proposition}{Proposition}[section]

\newtheorem{corollary}{Corollary}[section]

\newtheorem{theorem}{Theorem}[section]

\newtheorem{remark}{Remark}[section]

\newcommand{\PP}{\mathbb{P}}
\newcommand{\EE}{\mathbb{E}}

\DeclareMathOperator{\Cov}{Cov}

\makeatletter
\@addtoreset{equation}{section}
\makeatother

\renewcommand\thetable{\thesection.\@arabic\c@table}

\title[Amnesic Elephant Random Walks]{Limit theorems for amnesic elephant Random walks with random increment sizes}

\author{Cristian Coletti, Glauco Valle, Rafael Santos}
\thanks{Glauco Valle was supported by CNPq grant 307938/2022-0 and FAPERJ grant 
$\mbox{E-26/200.442/2023}$. Glauco Valle and Rafael Santos were supported by CNPq grant 403423/2023-6}

\address{
\newline
\newline
Cristian Coletti
\newline
UFABC -  Centro de Matemática, Computação e Cognição
\newline
e-mail: {\rm \texttt{cristian.coletti@ufabc.edu.br}} 
\newline
\newline
Glauco Valle 
\newline
UFRJ - Departamento de M\'etodos Estat\'{\i}sticos do Instituto de Matem\'atica.
\newline
e-mail: {\rm \texttt{glauco.valle@im.ufrj.br}}
\newline
\newline
Rafael Santos
\newline
UFRJ - Departamento de M\'etodos Estat\'{\i}sticos do Instituto de Matem\'atica.
\newline
e-mail: {\rm \texttt{rafaels@im.ufrj.br}}
}

\subjclass[2020]{60K37}
\keywords{amnesic elephant random walk, functional central limit theorem, non-Markovian random walks, stable limit theorem, superdiffusive fluctuations.}

\begin{document}

\maketitle

\begin{abstract}
We introduce a generalized amnesic elephant random walk in which, at each time, the direction of an underlying amnesic elephant random walk selects one of two increment distributions, while the actual increment is sampled from the corresponding independent sequence. The selected distribution does not need to determine the sign of the increment. This construction combines memory reinforcement, amnesia, and randomness in jump sizes. Assuming finite second moments, we establish central limit theorems and functional limit theorems, identifying diffusive, critical, and superdiffusive regimes determined by the memory and amnesia parameters. We further prove a Gaussian fluctuation theorem around the random superdiffusive limit, with an explicit limiting variance that captures the contributions of both the underlying amnesic walk and the random jump magnitudes. For the special case without amnesia, we also establish stable limit theorems and their functional counterparts when the increments belong to the domain of attraction of a non-Gaussian stable law, in the regime where heavy-tailed jumps dominate the memory contribution.
\end{abstract}


\setcounter{tocdepth}{2}




\section{Introduction}
\label{Introduction}

The elephant random walk (ERW) is a random walk model with long-range memory in its increments. Although its position is a time-inhomogeneous Markov process, its increments are generally non-Markovian. The model provides a simple framework for studying memory effects in complex systems (see \cite{S} and the references therein) and has become an important \textit{toy model} in physics and in the theory of stochastic processes. From a probabilistic point of view, one of its main features is the transition between regimes with Gaussian and non-Gaussian scaling limits. Related questions about fluctuations and scaling limits also arise in other models with dependent increments, such as self-repelling random walks with directed edges, see \cite{V}.

In this paper, we consider a generalization of the amnesic elephant random walk (AERW) introduced by Laulin \cite{LL}, allowing the distribution of each new increment to depend on the underlying AERW. Let us first recall the usual ERW, as studied in \cite{B18,C}. Its increments $(X_n)_{n\geq 1}$ are associated with a pair of probabilities $q,p\in(0,1)$ and are constructed as follows: Initially, we set $X_1$ as a Rademacher random variable with parameter $q$, i.e., $\PP(X_1 = 1) = q$ and $\PP(X_1 = -1) = 1-q$. Given $X_1,\ldots,X_n$, we choose an index $\beta_n$ uniformly from $\{1,\ldots,n\}$ and set $X_{n+1}=X_{\beta_n}$ with probability $p$, or $X_{n+1}=-X_{\beta_n}$ with probability $1-p$. Equivalently,
\[
X_{n+1}=\alpha_nX_{\beta_n},
\]
where $\PP(\alpha_n=1)=p$ and $\PP(\alpha_n=-1)=1-p$. At each time, the variables $\alpha_n$ and $\beta_n$ are independent of each other and of the past. The position of the walk is
\[
S_0=0,
\qquad
S_n=\sum_{j=1}^nX_j,\quad n\geq1.
\]

The limit theorems in \cite{B18,C} describe a change in behavior according to the parameter $p$. If $p<3/4$, a central limit theorem (CLT) holds on the diffusive scale $\sqrt{n}$. If $p=3/4$, the limit is still Gaussian, but the normalization becomes $\sqrt{n\log n}$. If $p>3/4$, then $S_n/n^{2p-1}$ converges almost surely to a non-degenerate, non-Gaussian random variable. Functional versions of these limit theorems were obtained in \cite{BB16,BO22}. In the superdiffusive regime, the non-Gaussian leading limit is accompanied by Gaussian fluctuations on a smaller scale: Kubota and Takei \cite{KT19} prove a CLT for the difference between the walk and its random leading term.

The AERW is constructed analogously, but the uniform choice of a past increment is replaced by a distribution that favors more recent increments. Let $\Gamma(\cdot)$ denote the gamma function. For a fixed parameter $\beta\geq0$, the memory index has probability mass function
\[
\PP(\beta_n=k)
=
\frac{(\beta+1)\Gamma(k+\beta)\Gamma(n)}
{\Gamma(k)\Gamma(n+\beta+1)},
\qquad 1\leq k\leq n,
\]
and the recurrence remains $X_{n+1}=\alpha_nX_{\beta_n}$. When $\beta=0$, this distribution is uniform and the AERW reduces to the classical ERW. Larger values of $\beta$ give greater preference to recent increments, while every past increment remains available for selection. Functional limit theorems for this model were obtained in \cite{LL}. The same three regimes occur, but the critical probability becomes
\[
p_c(\beta)=\frac{4\beta+3}{4\beta+4}.
\]
In the superdiffusive regime $p>p_c(\beta)$, the scale of $S_n$ is $n^\delta$, where
\[
\delta=(2p-1)(\beta+1)-\beta.
\]

Random step sizes have also been considered by Aguech \cite{A24}, together with gradually increasing memory and delays. Our proposal allows the entire distribution of a new increment, rather than only its sign, to depend on the underlying walk. Moreover, we use the AERW, so that the model includes both memory reinforcement and amnesia. Other related generalizations studied in the literature include random walks with reinforced increments \cite{BO22} and the minimal random walk \cite{CG22}.

More precisely, let $(X_n)_{n\geq1}$ be the increments of an AERW with parameters $q,p\in(0,1)$ and $\beta\geq0$. Let $(Z_n^+)_{n\geq1}$ and $(Z_n^-)_{n\geq1}$ be two sequences of i.i.d. random variables, independent of each other and of $(X_n)_{n\geq1}$. We define $W_0=0$ and
\[
W_{n+1}=W_n+\mathcal{Z}_{n+1},
\qquad n\geq0,
\]
where
\[
\mathcal{Z}_{n+1}:=Z_{n+1}^{X_{n+1}}
=
\begin{cases}
Z_{n+1}^+, & X_{n+1}=1,\\
Z_{n+1}^-, & X_{n+1}=-1.
\end{cases}
\]
The process $W=(W_n)_{n\geq0}$ is called the generalized amnesic elephant random walk (GAERW), or simply the generalized elephant random walk (GERW) when $\beta=0$.

The superscripts $+$ and $-$ label the two increment distributions and do not impose sign restrictions on $Z_n^+$ and $Z_n^-$. Thus $X_n$ selects the distribution of the increment, but need not determine its sign. The choices $Z_n^+\equiv1$ and $Z_n^-\equiv-1$ recover the AERW. More generally, choosing $Z_1^+$ and $-Z_1^-$ to have the same nonnegative distribution recovers a model with a common random jump magnitude and direction determined by $X_n$. Our construction also allows the two increment distributions to have different means, variances, and tail behavior.

When the first moments are finite, write
\[
\rho^+=\EE[Z_1^+],
\qquad
\rho^-=\EE[Z_1^-],
\qquad
\lambda=\frac{\rho^++\rho^-}{2},
\qquad
\rho=\frac{\rho^+-\rho^-}{2}.
\]
The decomposition
\[
W_n-n\lambda=\rho S_n+V_n,
\qquad
V_n=\sum_{k=1}^n
\bigl(Z_k^{X_k}-\lambda-\rho X_k\bigr),
\]
separates the contribution of the underlying AERW from the centered randomness of the increments. Conditionally on the entire AERW trajectory, the summands defining $V_n$ are independent and have mean zero. The relative sizes of these two contributions explain the different limiting behaviors considered in this paper.

In Section \ref{sec:GAERW}, we assume that the increment distributions have finite second moments. We prove a strong law of large numbers and establish functional limit theorems for the centered GAERW. The limits are Gaussian in the diffusive and critical regimes, with explicit covariance functions, whereas the superdiffusive functional limit is the random curve
\[
\bigl(\rho t^\delta\mathcal{L}_{q,\beta}\bigr)_{t\geq0},
\]
where $\mathcal{L}_{q,\beta}$ is the limiting random variable of the underlying AERW. When $\rho\neq0$, the corresponding one-dimensional normalizations are $\sqrt{n}$, $\sqrt{n\log n}$, and $n^\delta$, according as $p<p_c(\beta)$, $p=p_c(\beta)$, or $p>p_c(\beta)$.

In Section \ref{sec:fluctuations}, we study the fluctuations around the random superdiffusive limit. Under the same finite second moment assumption, we strengthen the convergence of $(W_n-n\lambda)/n^\delta$ to almost sure convergence and prove
\[
\frac{W_n-n\lambda-\rho n^\delta\mathcal{L}_{q,\beta}}
{\sqrt{n}}
\Longrightarrow
\mathcal{N}(0,\mathfrak{s}^2),
\qquad p>p_c(\beta).
\]
The variance $\mathfrak{s}^2$ is given explicitly in Theorem \ref{Fluctuations} and separates the contribution of the random increment sizes from that of the underlying AERW. This extends the Gaussian fluctuation theorem of \cite{KT19} to a model with both amnesia and random increment distributions. To the best of our knowledge, Gaussian fluctuations around the
random superdiffusive limit have not previously been established
for the AERW introduced in \cite{LL} with $\beta>0$.
Thus, our result is new even in the case
$Z_n^+\equiv1$ and $Z_n^-\equiv-1$. In particular, the centered randomness of the increments, which disappears at the leading superdiffusive scale, remains visible in the fluctuations. Related fluctuation results have also been obtained by
Bertenghi \cite{Be21} for superdiffusive step-reinforced random
walks with uniform sampling of the past, and by Roy, Takei
and Tanemura \cite{RTT25} for a unidirectional elephant random
walk with power-law memory, with Gaussian scale mixtures
arising in the latter case.

Finally, Section \ref{sec:stable} treats infinite variance increments for the GERW, that is, the case $\beta=0$. We assume that the two increment distributions belong to domains of attraction of stable laws with a common index $\alpha\in(0,2)$ and a common compatible normalization
\[
a_n\sim n^{1/\alpha}L(n),
\]
where $L$ is slowly varying. We prove convergence to an $\alpha$-stable law when
\[
p<\frac{\alpha+1}{2\alpha},
\]
and, for $1<\alpha<2$, establish the corresponding functional convergence to a stable L\'evy process in the Skorokhod $J_1$ topology.

For $1<\alpha<2$, we also describe the transition between limits governed by the heavy-tailed increments and limits governed by memory. Above the threshold $(\alpha+1)/(2\alpha)$, the memory contribution determines the leading limit whenever $\rho\neq$~$0$. At the threshold, the behavior depends on the slowly varying factor $L$: when $L(n)$ converges to a finite positive constant, the limit combines independent stable and ERW contributions when $\rho \neq 0$. These results distinguish the non-Gaussian behavior inherited from the elephant walk from that produced by the heavy tails of the increment distributions.

\section{Limit theorems for the GAERW with finite-variance increments}\label{sec:GAERW}

In this section, we assume that $(Z_n^+)_{n\ge 1}$ and $(Z_n^-)_{n\ge 1}$ are square-integrable
and we denote  
\begin{align*}
& (\sigma^+)^2 = \mbox{Var}[Z_1^+] \textrm{ and } (\sigma^-)^2 = \mbox{Var}[Z_1^-] .
\end{align*}
To simplify some expressions, we set $\vartheta = 2p-1$. Let $\mathcal{F}_n = \sigma(X_k,Z_k^+,Z_k^-: 1 \leq k \leq n)$  with $\mathcal{F}_0 = \{ \emptyset,\Omega\}$. We have as in \cite[(1.5)]{LL} that 
$$
\EE [ X_{n+1} | \mathcal{F}_n ] = \vartheta (\beta+1) \overline X^\mu_n \ \, a.s., \ \, n\ge 1,
$$
where
$$
\mu_{n} = \frac{\Gamma(n+\beta)}{\Gamma(n) \Gamma(\beta+1)} \ \mbox{and} \ \overline X^\mu_n = \frac{1}{n \mu_{n+1}} \sum_{k=1}^n X_k \mu_k, 
$$
which is simply the sample mean $\overline X_n$ when $\beta=0$ as in \cite[(2.2)]{B18}. We want to compute the expectations of the increments $\mathcal{Z}_{n+1}= W_{n+1} - W_{n}$. Note that
\begin{align}\label{incr-form}
\mathcal{Z}_{n+1} &= Z_{n+1}^{X_{n+1}} = \frac{(1 + X_{n+1})}{2} Z_{n+1}^+ + \frac{(1 - X_{n+1})}{2} Z_{n+1}^-
\nonumber \\
&= Z_{n+1}^{\alpha_n X_{\beta_n}} = \frac{(1 + \alpha_n X_{\beta_n})}{2} Z_{n+1}^+ + \frac{(1 - \alpha_n X_{\beta_n})}{2} Z_{n+1}^- ,
\end{align}
thus 
\begin{align*}
\EE [{\mathcal{Z}}_{n+1}|{\mathcal{F}}_n] &= \EE \Big[ \frac{(1+{X}_{n+1})}{2} Z_{n+1}^{+} + \frac{(1-{X}_{n+1})}{2} Z_{n+1}^{-} \Big| {\mathcal{F}}_n\Big]\\[0.5em]
&= \frac{\EE[Z_{n+1}^{+}]+E[Z_{n+1}^{-}]}{2} + \frac{\EE[{X}_{n+1}|{\mathcal{F}}_n] \EE[Z_{n+1}^+]}{2} - \frac{\EE[{X}_{n+1}|{\mathcal{F}}_n] \EE[Z_{n+1}^-]}{2}\\[0.5em]
&= \lambda + \rho\vartheta(\beta+1) \overline X^\mu_n.
\end{align*}
Therefore
\[
\EE[W_{n+1}|\mathcal{F}_n] = W_n + \EE[\mathcal{Z}_{n+1}|\mathcal{F}_n] = {W}_n + \lambda + \rho\vartheta(\beta+1) \overline X^\mu_n.
\]
From \cite[(4.5)]{LL} we have that $\overline X_n^{\mu}$ converges almost surely to zero. Thus $W_n - n \lambda$ has asymptotically centered increments. We use this to prove a strong law of large numbers for $W$.

\begin{lemma}\label{LLN}
$$
\frac{W_n}{n} \xrightarrow[n\to \infty]{a.s.} \lambda .
$$
\end{lemma}

\begin{proof} Suppose without loss of generality that $\lambda = 0$. Write
\begin{align*}
\frac{{W}_n}{n} &= \frac{1}{n} \sum_{j=1}^n ({W}_j - {W}_{j-1})\\[0.5em]
&= \frac{1}{n} \sum_{j=1}^n ({W_j} - \EE[{W_j}|{\mathcal{F}}_{j-1}]) + \frac{1}{n} \sum_{j=2}^n \frac{\rho \vartheta(\beta+1)}{(j-1)\mu_{j}} \sum_{k=1}^{j-1} {X}_k \mu_k + \frac{\rho(2q-1)}{n}.
\end{align*}
The third term above converges almost surely to zero. For the second term above, we have by  \cite[(4.5)]{LL} that
\[
\lim_{j \to \infty} \frac{1}{(j-1)\mu_{j-1}} \sum_{k=1}^{j-1} {X}_k\mu_k =0 \ a.s.,
\]
which implies that the aforementioned second term also converges almost surely to zero since $\frac{\mu_{j-1}}{\mu_j} = \frac{j-1}{j-1+\beta} \xrightarrow[]{n \to \infty} 1$. Now set ${M}_0 = 0$ and
\[
{M}_n = \sum_{j=1}^n ({W}_j - \EE[{W}_j|{\mathcal{F}}_{j-1}]), \ n \geq 1.
\]
Then $(M_n)_{n \geq 0}$ is a mean-zero $(\mathcal{F}_n)$-martingale, with
\[
\Delta M_j = {W}_j - \EE[{W}_j|{\mathcal{F}}_{j-1}]
\]
and therefore
\begin{align*}
\EE [ (\Delta {M}_j)^2 | {\mathcal{F}}_{j-1} ]
& = \EE [  ({W}_j - \EE[{W}_j |{\mathcal{F}}_{j-1}])^2 | {\mathcal{F}}_{j-1} ] = \EE [  ({\mathcal{Z}}_j - \EE[{\mathcal{Z}}_j |{\mathcal{F}}_{j-1}])^2 | {\mathcal{F}}_{j-1} ]\\
&= \EE [ {\mathcal{Z}}_j^2 | {\mathcal{F}}_{j-1} ] - \EE[{\mathcal{Z}}_j |{\mathcal{F}}_{j-1}]^2 \leq \EE [ (Z_{1}^+)^2 ] \vee \EE [ (Z_{1}^-)^2 ],
\end{align*}
implying that 
\begin{equation}
\label{eq:lemma1}
\EE [ (\Delta {M}_j)^2 | {\mathcal{F}}_{j-1} ] \leq E_{max}^2,
\end{equation}
where $E^2_{max} = \EE [ (Z_{1}^+)^2 ] \vee \EE [ (Z_{1}^-)^2 ]$. Taking expectations on both sides of \eqref{eq:lemma1} gives that
\[
\sup_{j\geq1} \EE[(\Delta M_j)^2] \leq E^2_{max} < \infty.
\]
Now consider the martingale
\[
\tilde{M}_n = \sum_{j=1}^n \frac{\Delta M_j}{j}.
\]
Since $(M_n)_{n \geq 1}$ is a square-integrable martingale, its increments are orthogonal in $L^2$ and hence
\[
\sup_{n \geq 1} \EE[\tilde{M}_n^2] \leq E^2_{max} \sum_{j=1}^n \frac{1}{j^2} < \infty.
\]
Thus $(\tilde{M}_n)_{n \geq 1}$ is bounded in $L^2$ and consequently there exists $\tilde{M}_{\infty} \in L^2$ such that $\tilde{M}_n \xrightarrow[n\to \infty]{a.s.} \tilde{M}_{\infty}$. Therefore, by Kronecker's Lemma,
\[
\frac{M_n}{n} = \frac{1}{n} \sum_{j=1}^n \Delta M_j \xrightarrow[n\to \infty]{a.s.} 0,
\]
which concludes the proof.

\end{proof}

Concerning the functional limit theorems, the GERW (when $\beta=0$) inherits the usual diffusive, critical, and superdiffusive regimes of the ERW, corresponding respectively to $p<3/4, p=3/4$, and $p>3/4$; see \cite{BB16}. For the AERW, and consequently for the GAERW, the critical probability is shifted to
\[
p_c(\beta)=\frac{4\beta+3}{4(\beta+1)},
\]
as shown in \cite{LL}. There is, however, an exceptional value within the diffusive regime, namely
\[
\tilde{p}(\beta)=\frac{2\beta+1}{2(\beta+1)}<p_c(\beta).
\]
At this value, the martingale decomposition used in \cite{LL} becomes singular and its functional limit theorem cannot be applied directly. The exceptional case is covered by \cite[Theorem 2.11]{MM} and the normalization remains diffusive, but the covariance function of the limiting Gaussian process contains a logarithmic term. We therefore state separately the functional limit theorem for the generic diffusive case $p\neq \tilde{p}(\beta)$ and for the exceptional case $p=\tilde{p}(\beta)$. Throughout the entire paper, $D([0,\infty))$ is endowed with the Skorokhod \(J_1\) topology.

\medskip

\begin{theorem}\label{FCLT<} If $p < \frac{4\beta + 3}{4\beta + 4}$, $p \neq \tilde{p}(\beta)$ and $\max\{(\sigma^+)^2, (\sigma^-)^2, \rho^2\} > 0$ then
$$
\Big( \frac{{W}_{\lfloor nt \rfloor} - \lfloor nt \rfloor\lambda}{\sqrt{n}} \Big)_{t\geq 0}
$$
converges in distribution on $D([0,\infty))$ to a continuous mean-zero Gaussian process $Y=(Y_t)_{t \ge 0}$ starting at $Y_0 \equiv 0$, whose covariance is, for $0 < s \leq t$,
\begin{align}\label{covfunction}
\Cov(Y_s,Y_t) &= \frac{s(\sigma^+)^2 + s(\sigma^-)^2}{2} + \frac{s \rho^2 \vartheta(1+2\beta-\vartheta-\vartheta\beta)}{(1-\vartheta)(1+2\beta-2\vartheta-2\vartheta\beta) (\vartheta-\beta+\vartheta\beta)}\Big( \frac{t}{s}\Big)^{\vartheta-\beta(1-\vartheta)} \nonumber\\
&+ \frac{s \rho^2 \beta}{(1-\vartheta)(\beta-\vartheta-\vartheta\beta)}.
\end{align}
\end{theorem}

\medskip

\begin{theorem}\label{FCLT<ptilde} If $p = \tilde{p}(\beta)$ and $\max\{(\sigma^+)^2, (\sigma^-)^2, \rho^2\} > 0$ then
$$
\Big( \frac{{W}_{\lfloor nt \rfloor} - \lfloor nt \rfloor\lambda}{\sqrt{n}} \Big)_{t\geq 0}
$$
converges in distribution on $D([0,\infty))$ to a continuous mean-zero Gaussian process $\tilde{Y}=(\tilde{Y}_t)_{t \ge 0}$ starting at $\tilde{Y}_0 \equiv 0$, whose covariance is, for $0 < s \leq t$,
$$
\Cov(\tilde{Y}_s,\tilde{Y}_t) = \frac{s(\sigma^+)^2 + s(\sigma^-)^2}{2} + \rho^2s \Big[ 1+ 2\beta  + 2\beta^2 + \beta(\beta+1) \log\Big( \frac{t}{s}\Big)\Big].\nonumber
$$
\end{theorem}

\medskip

When $\beta = 0$, the covariance functions that appear in Theorems \ref{FCLT<} and \ref{FCLT<ptilde} coincide, having the following simplified expression, for $0 < s \leq t$,
$$
\mbox{Cov}(Y_s,Y_t) = \frac{s(\sigma^+)^2 + s(\sigma^-)^2}{2} + \frac{s \rho^2}{3 - 4p}\Big( \frac{t}{s}\Big)^{2p-1}.
$$
This includes noise reinforced Brownian motion as a special case when both increment distributions of the GERW are degenerate, i.e. $\sigma^+ = \sigma^- = 0$. 

Evaluating the limiting process at $t=1$ yields the following Central Limit Theorem.

\begin{corollary}\label{prop:unidim} If $p < \frac{4\beta + 3}{4\beta + 4}$ and $\max\{(\sigma^+)^2, (\sigma^-)^2, \rho^2\} > 0$, then
$$
\frac{W_{n} - n\lambda}{\sqrt{n}} \Longrightarrow \mathcal{N}\Big(0, \frac{(\sigma^+)^2 + (\sigma^-)^2}{2} + \frac{\rho^2(\beta-p+1)}{(1-p)(4 \beta-4p \beta-4p+3)} \Big) , \ \textrm{ as } n \to \infty.
$$
\end{corollary}

\medskip

Note that Corollary \ref{prop:unidim} also covers the exceptional case of $p = \tilde{p}(\beta)$. In the case $\beta=0$, the limiting distribution in Corollary \ref{prop:unidim} is 
$$
\mathcal{N}\Big(0, \frac{(\sigma^+)^2 + (\sigma^-)^2}{2} + \frac{\rho^2}{3-4p} \Big).
$$
The first step in the proof of Theorem \ref{FCLT<} is the convergence of the finite-dimensional distributions, which is the content of the next proposition.

\begin{proposition}\label{prop:finitedim<} Let $0 < s \leq t < \infty$. If $p < \frac{4\beta + 3}{4\beta + 4}$, $p \neq \tilde{p}(\beta)$ and $\max\{(\sigma^+)^2, (\sigma^-)^2, \rho^2\} > 0$, then
$$
\left(\frac{W_{\lfloor ns \rfloor} - \lfloor ns \rfloor\lambda}{\sqrt{n}},\frac{W_{\lfloor nt \rfloor} - \lfloor nt \rfloor\lambda}{\sqrt{n}}  \right)\Longrightarrow \mathcal{N}\Big(\mathbf{0}, \Sigma_{s,t} \Big) , \ \textrm{ as } n \to \infty,
$$
\noindent where $\Sigma_{s,t}$ is a covariance matrix of $(Y_s,Y_t)$ for $Y$ as in the statement of Theorem \ref{FCLT<}.
\end{proposition}

\begin{proof} 
To begin, we have $W_n = \sum_{k=1}^n \mathcal{Z}_k$ and note that the increments $\mathcal{Z}_k$ can be written in the following way:
\begin{align}\label{decompZ}
\mathcal{Z}_k & = \frac{1}{2} \Big[(1 + X_k)Z_{k}^+  + (1 - X_k) Z_{k}^-\Big] = \frac{1}{2} \Big[Z_{k}^+ + X_kZ_{k}^+  +  Z_{k}^- - X_k Z_{k}^-\Big] \nonumber \\
& = \frac{1}{2} \Big[Z_{k}^+ +  Z_{k}^- + X_k(Z_{k}^+ - \rho^+) + \rho^+X_k   - X_k (Z_{k}^- - \rho^-) - \rho^- X_k\Big] \nonumber \\
& = \frac{1}{2} \Big[X_k(\rho^+ - \rho^- ) + Z_{k}^+ + Z_{k}^-  + X_k(Z_{k}^+ - \rho^+) - X_k(Z_k^- - \rho^-)\Big] \nonumber \\
& = \frac{1}{2} \Big[X_k(\rho^+ - \rho^- ) + \tilde{Z}_k^+  + X_k\tilde{Z}_k^+ + \tilde{Z}_k^- - X_k\tilde{Z}_k^-+ (\rho^+ + \rho^-)\Big] \nonumber \\
& = \frac{1}{2} \Big[X_k(\rho^+ - \rho^- ) + \tilde{Z}_k^+(X_k +1) - \tilde{Z}_k^-(X_k - 1) + (\rho^+ + \rho^-)\Big],
\end{align}
where $\tilde{Z}_k^+ = Z_k^+ - \rho^+$ and $\tilde{Z}_k^- = Z_k^- - \rho^-$. Now for any fixed $0 < s \leq t < \infty$, $(\theta_1,\theta_2) \in \mathbb{R}^2$ and $n \in \mathbb{N}$, the joint characteristic function of $\left(\frac{W_{\lfloor ns \rfloor} - \lfloor ns \rfloor\lambda}{\sqrt{n}},\frac{W_{\lfloor nt \rfloor} - \lfloor nt \rfloor\lambda}{\sqrt{n}}  \right)$, which we denote by $\psi_{n,s,t}(\theta_1,\theta_2) = \psi_{n}(\theta_1,\theta_2)$ can be derived in the following way:
\begin{align*}
 \psi_n(\theta_1,\theta_2) = & \EE \Big[ \exp\Big({\textstyle \frac{i\theta_1 (W_{\lfloor ns \rfloor} - \lfloor ns \rfloor\lambda)}{\sqrt{n}} + \frac{i\theta_2 (W_{\lfloor nt \rfloor} - \lfloor nt \rfloor\lambda)}{\sqrt{n}}}\Big)\Big]\\
= & \EE \Big[ \exp \Big({\textstyle \frac{i\theta_1}{2} {\displaystyle \sum_{k=1}^{\lfloor ns \rfloor}} \Big[\frac{X_k}{\sqrt{n}}(\rho^+ - \rho^- ) + \frac{\tilde{Z}_k^+(X_k +1) - \tilde{Z}_k^-(X_k - 1) + (\rho^+ + \rho^-)}{{\sqrt{n}}} - \frac{\rho^+ + \rho^-}{\sqrt{n}}}\Big]\\
& {\textstyle + \frac{i\theta_2}{2}  {\displaystyle \sum_{k=1}^{\lfloor nt \rfloor}} \Big[\frac{X_k}{\sqrt{n}}(\rho^+ - \rho^- ) + \frac{\tilde{Z}_k^+(X_k +1) - \tilde{Z}_k^-(X_k - 1) + (\rho^+ + \rho^-)}{{\sqrt{n}}} - \frac{\rho^+ + \rho^-}{\sqrt{n}}\Big]\Big)\Big]}\\
= & \EE \Big[ \exp \Big( {\textstyle i \rho \theta_1 {\displaystyle \sum_{k=1}^{\lfloor ns \rfloor}} \frac{X_k}{\sqrt{n}} + i \rho \theta_2 {\displaystyle\sum_{k=1}^{\lfloor nt \rfloor}} \frac{X_k}{\sqrt{n}} } \Big) \exp \Big( {\textstyle \frac{i \theta_1}{2} {\displaystyle \sum_{k=1}^{\lfloor ns \rfloor}} \frac{\tilde{Z}_k^+(X_k +1) - \tilde{Z}_k^-(X_k - 1)}{{\sqrt{n}}} }\\
& \qquad {\textstyle + \frac{i \theta_2}{2} {\displaystyle \sum_{k=1}^{\lfloor ns \rfloor}} \frac{\tilde{Z}_k^+(X_k +1) - \tilde{Z}_k^-(X_k - 1)}{{\sqrt{n}}} + \frac{i \theta_2}{2} {\displaystyle \sum_{k=\lfloor ns \rfloor + 1}^{\lfloor nt \rfloor}} \frac{\tilde{Z}_k^+(X_k +1) - \tilde{Z}_k^-(X_k - 1)}{{\sqrt{n}}}} \Big) \Big].
\end{align*}
Let $\mathcal{G}_n = \sigma(X_k; k \leq n)$ be the filtration generated by the AERW and $\mathcal{G}_{\infty} = \sigma(X_1,X_2,...)$. At this point, it is convenient to condition on $\mathcal{G}_{\infty}$ to separate the expectations and hence $\psi_n(\theta_1,\theta_2)$ is equal to
\begin{align*}
& \EE \Big[ \exp \Big( {\textstyle i \rho \theta_1 {\displaystyle \sum_{k=1}^{\lfloor ns \rfloor}} \frac{X_k}{\sqrt{n}} + i \rho \theta_2 {\displaystyle \sum_{k=1}^{\lfloor nt \rfloor}} \frac{X_k}{\sqrt{n}} \Big) \times} \\ 
& \quad \EE \Big[\exp\Big( {\textstyle \frac{i(\theta_1 + \theta_2)}{2} {\displaystyle \sum_{k=1}^{\lfloor ns \rfloor}} \frac{\tilde{Z}_k^+(X_k +1) - \tilde{Z}_k^-(X_k - 1)}{{\sqrt{n}}} + \frac{i \theta_2}{2} {\displaystyle \sum_{k=\lfloor ns \rfloor+1}^{\lfloor nt \rfloor}} \frac{\tilde{Z}_k^+(X_k +1) - \tilde{Z}_k^-(X_k - 1)}{{\sqrt{n}}}} \Big) \Big| {\textstyle \mathcal{G}_{\infty}}\Big] \Big].
\end{align*}
Note that the random variables $\tilde{Z}_k^+(X_k +1)$ and $\tilde{Z}_k^-(X_k - 1)$ are conditionally independent given $\mathcal{G}_{\infty}$, for any $k \in \mathbb{N}$. Then $\psi_n(\theta_1,\theta_2)$ is given by
\begin{align}\label{phi_1}
\EE \Big[ \exp \Big( & {\textstyle i \rho \theta_1 {\displaystyle \sum_{k=1}^{\lfloor ns \rfloor}} \frac{X_k}{\sqrt{n}} + i \rho \theta_2 {\displaystyle \sum_{k=1}^{\lfloor nt \rfloor}} \frac{X_k}{\sqrt{n}} } \Big) \nonumber \\
& \psi^+_{n,s,X}(\theta_1 + \theta_2) \psi^-_{n,s,X}(\theta_1 + \theta_2) \psi^+_{n,t-s,X}(\theta_2) \psi^-_{n,t-s,X}(\theta_2)  \Big],
\end{align}
where 
\begin{align*}
\psi^+_{n,s,X}(\theta_1 + \theta_2) & =  \EE \Big[\exp\Big( {\textstyle \frac{i(\theta_1 + \theta_2)}{2} {\displaystyle \sum_{k=1}^{\lfloor ns \rfloor}} \frac{\tilde{Z}_k^+(X_k +1)}{{\sqrt{n}}} } \Big) \Big| {\textstyle \mathcal{G}_{\infty}} \Big],\\
\psi^-_{n,s,X}(\theta_1 + \theta_2) & =  \EE \Big[\exp\Big( {\textstyle - \frac{i(\theta_1 + \theta_2)}{2} {\displaystyle \sum_{k=1}^{\lfloor ns \rfloor}} \frac{\tilde{Z}_k^-(X_k -1)}{{\sqrt{n}}} } \Big) \Big| {\textstyle \mathcal{G}_{\infty} }\Big],\\
\psi^+_{n,t-s,X}(\theta_2) & =  \EE \Big[\exp\Big( {\textstyle \frac{i\theta_2}{2} {\displaystyle \sum_{k=\lfloor ns \rfloor + 1}^{\lfloor nt \rfloor}} \frac{\tilde{Z}_k^+(X_k +1)}{{\sqrt{n}}} } \Big) \Big| {\textstyle \mathcal{G}_{\infty} } \Big]
\end{align*}
and
$$
\psi^-_{n,t-s,X}(\theta_2) =  \EE \Big[\exp\Big( {\textstyle - \frac{i\theta_2}{2} {\displaystyle \sum_{k=\lfloor ns \rfloor + 1}^{\lfloor nt \rfloor}} \frac{\tilde{Z}_k^-(X_k -1)}{{\sqrt{n}}} } \Big) \Big| {\textstyle \mathcal{G}_{\infty} } \Big].
$$
Given a realization $X = (x_1,x_2,...) \in \{-1,1\}^\mathbb{N}$, set $V_k = (x_k + 1) \tilde{Z}_k^+$. $(V_k)_{k\ge 1}$ is a sequence of independent random variables such that
$$
\operatorname{Var}(V_k) = (x_k + 1)^2 (\sigma^+)^2 \le 4 (\sigma^+)^2.
$$
Moreover
$$
\frac{\operatorname{Var}(\sum_{k=1}^{\lfloor ns \rfloor} V_k)}{n} = \frac{(\sigma^+)^2}{n} \sum_{k=1}^{\lfloor ns \rfloor} (x_k + 1)^2 = \frac{2 \lfloor ns \rfloor (\sigma^+)^2}{n} + \frac{2}{n} \sum_{k=1}^{\lfloor ns \rfloor} x_k.
$$
The last term above converges to zero almost surely by \cite[Theorem 2.1]{LL} and hence, almost surely,
$$
\lim_{n\rightarrow \infty} \operatorname{Var} \Big( \frac{1}{2} \sum_{k=1}^{\lfloor ns \rfloor} \frac{\tilde{Z}_k^+(X_k + 1)}{{\sqrt{n}}} \Big| \mathcal{G}_{\infty} \Big) = \frac{s(\sigma^+)^2}{2},
$$
and, for any $\epsilon > 0$,
$$
\sum_{k=1}^{\lfloor ns \rfloor} \EE \Big[ \Big(\frac{\tilde{Z}_k^+(X_k + 1)}{{2\sqrt{n}}}\Big)^2\mathbf{1}_{\{|\tilde{Z}_k^+(X_k + 1)|> 2 \epsilon \sqrt{n}\}} \Big| \mathcal{G}_{\infty} \Big] \leq \frac{\lfloor ns \rfloor}{n} \EE \big[ (\tilde{Z}_1^+)^2 \mathbf{1}_{\{|\tilde{Z}_1^+| > \epsilon \sqrt{n}\}}\big],
$$
which converges to zero as n goes to infinity since $\tilde{Z}_1^+$ is a square-integrable random variable. Thus, the Lindeberg–Feller Central Limit Theorem yields
$$
\lim_{n\rightarrow \infty} \psi^+_{n,s,X}(\theta_1 + \theta_2) = e^{-\frac{1}{4}\big[(\theta_1 + \theta_2) \sigma^+\sqrt{s}  \big]^2} \quad a.s.
$$
Analogously
\begin{align*}
\lim_{n\rightarrow \infty} \psi^-_{n,s,X}(\theta_1 + \theta_2) & = e^{-\frac{1}{4}\big[(\theta_1 + \theta_2) \sigma^-\sqrt{s}  \big]^2} \quad a.s.,\\
\lim_{n\rightarrow \infty} \psi^+_{n,t-s,X}(\theta_2) & = e^{-\frac{1}{4}\big[\theta_2 \sigma^+\sqrt{t-s}  \big]^2} \quad a.s.,\\
\lim_{n\rightarrow \infty} \psi^-_{n,t-s,X}(\theta_2) & = e^{-\frac{1}{4}\big[\theta_2 \sigma^-\sqrt{t-s}  \big]^2} \quad a.s.
\end{align*}
Then, using the dominated convergence theorem in \eqref{phi_1},
\begin{align}\label{phi_2}
\lim_{n\rightarrow \infty} \psi_n(\theta_1,\theta_2) & = \exp\left\{ - \frac{(\theta_1 + \theta_2)^2s}{2} \left[ \frac{(\sigma^{+})^2 + (\sigma^{-})^2}{2} \right]\right\} \exp\left\{ - \frac{\theta_2^2(t-s)}{2} \left[ \frac{(\sigma^{+})^2 + (\sigma^{-})^2}{2} \right]\right\} \nonumber\\
& \times \lim_{n\rightarrow \infty} \EE \Big[ \exp \Big( i \rho \theta_1 \sum_{k=1}^{\lfloor ns \rfloor} \frac{X_k}{\sqrt{n}} + i \rho \theta_2 \sum_{k=1}^{\lfloor nt \rfloor} \frac{X_k}{\sqrt{n}} \Big) \Big].
\end{align}

Now we have to deal with the last term in \eqref{phi_2}. From \cite[Theorem 2.3]{LL} it follows that the last term in \eqref{phi_2} is equal to
\begin{align*}
\exp\Big\{&-\frac{\rho^2}{2(1-\vartheta)}\Big[ \frac{\theta_1^2 s (2\beta+1-\vartheta)}{1+2\beta-2\vartheta-2\vartheta\beta} + \frac{2\theta_1 \theta_2 s \vartheta(1+2\beta-\vartheta-\vartheta\beta)}{(1+2\beta-2\vartheta-2\vartheta\beta)(\vartheta-\beta+\vartheta\beta)}\Big(\frac{t}{s}\Big)^{\vartheta-\beta(1-\vartheta)}\Big.\Big.\\[0.8em]
&+ \Big. \Big. \frac{2\theta_1 \theta_2 s \beta}{\beta-\vartheta-\vartheta\beta}+\frac{\theta_2^2 t (2\beta+1-\vartheta)}{1+2\beta-2\vartheta-2\vartheta\beta} \Big] \Big\},
\end{align*}
\noindent and going back to \eqref{phi_2}, we have that
\begin{align*}
& \lim_{n\rightarrow \infty} \psi_n(\theta_1,\theta_2)  = \exp\left\{-\frac{1}{2}\left[ \theta_1^2 \left(\frac{s(\sigma^+)^2+s(\sigma^-)^2}{2} + \frac{s\rho^2(2\beta+1-\vartheta)}{(1-\vartheta)(1+2\beta-2\vartheta-2\vartheta\beta)} \right)\right.\right.\\[0.8em]
& + 2\theta_1\theta_2\left(\frac{s(\sigma^+)^2+s(\sigma^-)^2}{2} + \frac{s\rho^2\vartheta(1+2\beta-\vartheta-\vartheta\beta)(1-\vartheta)^{-1}}{(1+2\beta-2\vartheta-2\vartheta\beta)(\vartheta-\beta+\vartheta\beta)}\Big(\frac{t}{s}\Big)^{\vartheta-\beta(1-\vartheta)}\right.\\[0.8em]
& + \left.\frac{s\rho^2\beta(1-\vartheta)^{-1}}{(\beta-\vartheta-\vartheta\beta)}\right) + \theta_2^2 \left. \left.\left(\frac{t(\sigma^+)^2+t(\sigma^-)^2}{2} + \frac{t\rho^2(2\beta+1-\vartheta)}{(1-\vartheta)(1+2\beta-2\vartheta-2\vartheta\beta)}\right)\right] \right\},
\end{align*}
which concludes the proof.
\end{proof}

\medskip

The next step in the proof of Theorem \ref{FCLT<} is to verify tightness, which is addressed in the following proposition. Its statement and proof are also valid for $p = \tilde{p}(\beta)$.

\begin{proposition}
\label{prop:tightness}
If $p < \frac{4\beta+3}{4\beta+4}$ and $\max\{(\sigma^+)^2,(\sigma^-)^2,\rho^2\} > 0$, then
\[
Y^{(n)}_{t} = \frac{W_{\lfloor nt \rfloor} - \lfloor nt \rfloor\lambda}{\sqrt{n}}, \ t\ge 0
\]
is tight on $D([0,\infty))$.
\end{proposition}

\begin{proof}
Using again the decomposition \eqref{decompZ}, write
\begin{equation}\label{eqprop22}
Y^{(n)}_{t} = \frac{(\rho^+ - \rho^- )}{2\sqrt{n}} \sum_{k=1}^{\lfloor nt \rfloor} X_k + \frac{1}{2} \sum_{k=1}^{\lfloor nt \rfloor} \frac{\tilde{Z}_k^+(X_k +1)}{{\sqrt{n}}} - \frac{1}{2} \sum_{k=1}^{\lfloor nt \rfloor} \frac{\tilde{Z}_k^-(X_k - 1)}{{\sqrt{n}}} , \ \forall \, t\ge 0. 
\end{equation}
From \cite[Theorem 2.3]{LL}, the first term of \eqref{eqprop22} converges in distribution if $p < \frac{4\beta+3}{4\beta+4}$ and $p \neq \tilde{p}(\beta)$. From \cite[Theorem 2.11]{MM}, the same conclusion holds if $p = \tilde{p}(\beta)$. Although \cite{MM} assumes $q = 1/2$ for the Rademacher parameter that determines the first increment of the AERW, the same conclusion holds for every $q \in (0,1)$ since the martingale functional central limit theorem used in the proof of \cite[Theorem 2.11]{MM} can be applied by conditioning on $X_1 = 1$ or $X_1 = -1$, with the same limit in both cases.  Since the limiting processes in both results have continuous sample paths, the corresponding convergence also holds with respect to the uniform topology on compact time intervals. Consequently, the first term of \eqref{eqprop22} is tight with respect to this topology. \

Hence, it remains only to prove that the second and third terms of (2.7) are tight with respect to the uniform topology on compact time intervals. We treat the second term, since the argument for the third one is analogous. Set
\[
J_t^{(n)}
=\frac{1}{\sqrt n}\sum_{k=1}^{\lfloor nt\rfloor}
\tilde Z_k^+(X_k+1),
\qquad t\geq0.
\]
We shall in fact establish the functional convergence of $(J^{(n)}_t)_{t \ge 0}$. Recall that $\mathcal{G}_k = \sigma(X_j: 1 \leq j \leq k)$ and consider the enlarged filtration $\mathcal{H}^+_k = \mathcal{G}_{\infty} \vee \sigma(\tilde{Z}_j^+: 1 \leq j \leq k)$, where $\mathcal{H}_0^+ = \mathcal{G}_{\infty}$. Conditionally on $\mathcal{G}_{\infty}$, the process $J_t^{(n)}$ is a discrete-time  square-integrable martingale with increments
\[
D_k^{(n)} = \frac{1}{\sqrt{n}} \tilde{Z}_k^{+}(X_k+1).
\]
Its predictable quadratic variation is
\begin{align*}
\langle J^{(n)}\rangle_t &= \sum_{k=1}^{\lfloor nt\rfloor} \EE\big[(D_k^{(n)})^2\big|\mathcal{H}_{k-1}^+\big] = \frac{1}{n} \sum_{k=1}^{\lfloor nt\rfloor} \EE\big[(\tilde{Z}_k^+)^2(X_k+1)^2\big|\mathcal{H}_{k-1}^+\big] = \frac{(\sigma^+)^2}{n}\sum_{k=1}^{\lfloor nt\rfloor} (X_k+1)^2\\
&= \frac{2(\sigma^+)^2}{n} \Big( \lfloor nt \rfloor + \sum_{k=1}^{\lfloor nt \rfloor} X_k\Big).
\end{align*}
Since $\frac{1}{n}\sum_{k=1}^{\lfloor nt \rfloor} X_k$ converges to zero almost surely by \cite[Theorem 2.1]{LL}, we have, for every $T > 0$,
\[
\max_{1 \leq m \leq nT} \frac{1}{n}\Big| \sum_{k=1}^m X_k\Big| \xrightarrow[]{n\to \infty} 0, \ a.s.
\]
and hence,
\[
\sup_{0 \leq t \leq T}\big|\langle J^{(n)}\rangle_t - 2(\sigma^+)^2t\big|  \to 0, \ \mbox{a.s.}
\]
Moreover, for every $\epsilon > 0$,
\begin{align*}
\sum_{k=1}^{\lfloor nT\rfloor}\EE\big[(D_k^{(n)})^2\mathbf{1}_{\{|D_k^{(n)}|>\epsilon\}}\big| \mathcal{H}_{k-1}^+\big] &= \sum_{k=1}^{\lfloor nT\rfloor}\frac{(X_k+1)^2}{n} \EE \big[ (\tilde{Z}_k^+)^2 \mathbf{1}_{\{|\tilde{Z}_k^+||X_k+1|> \epsilon \sqrt n\}} \big| \mathcal{H}_{k-1}^+\big]\\
&\leq \frac{4 \lfloor nT\rfloor}{n} \EE \big[ (\tilde{Z}_1^+)^2 \mathbf{1}_{\{2|\tilde{Z}_1^+|> \epsilon \sqrt n\}} \big].
\end{align*}
The right-hand side converges to zero as $n$ goes to infinity by the square integrability of $\tilde{Z}_1^+$, proving the Lindeberg condition. Hence, by the martingale functional central limit theorem (see \cite{W07}, for instance),
\[
\bigl(J_t^{(n)}\bigr)_{t\geq 0}
\Longrightarrow
\bigl(\sqrt{2}\,\sigma^+B_t\bigr)_{t\geq 0}
\]
in $D([0,\infty))$, where $(B_t)_{t\geq 0}$ is a standard Brownian motion. Since the limiting process has continuous sample paths, convergence in the $J_1$ topology also implies convergence in the topology of uniform convergence on compact time intervals. Consequently, the second term in \eqref{eqprop22} is tight with respect to the uniform topology on compact time intervals. The same argument applies to the third term. Since the first term has already been shown to be tight with respect to the same topology, all three components in \eqref{eqprop22} are tight in the uniform topology on compact time intervals. Finally, since addition is continuous with respect to this topology, $(Y_t^{(n)})_{t\geq 0}$ is tight in the uniform topology on compact time intervals, and hence it is tight in $D([0,\infty))$ endowed with the $J_1$ topology.
\end{proof}

\begin{proof}[Proof of Theorem \ref{FCLT<}]
It is straightforward to obtain the convergence of the finite-dimensional distributions analogously to the two-dimensional case from Proposition \ref{prop:finitedim<}. The convergence of the finite-dimensional distributions together with tightness from Proposition \ref{prop:tightness} concludes the proof.
\end{proof}

\medskip

To prove Theorem \ref{FCLT<ptilde} we already have tightness from Proposition \ref{prop:tightness}, but we still need to obtain the convergence of the finite-dimensional distributions, since the result obtained in Proposition \ref{prop:finitedim<} is not valid for $p = \tilde{p}(\beta)$. It will be addressed in the next proposition.

\begin{proposition}\label{prop:finitedim<tilde} Let $0 < s \leq t < \infty$. If $p = \tilde{p}(\beta)$ and $\max\{(\sigma^+)^2, (\sigma^-)^2, \rho^2\} > 0$, then
$$
\left(\frac{W_{\lfloor ns \rfloor} - \lfloor ns \rfloor\lambda}{\sqrt{n}},\frac{W_{\lfloor nt \rfloor} - \lfloor nt \rfloor\lambda}{\sqrt{n}}  \right)\Longrightarrow \mathcal{N}\Big(\mathbf{0}, \tilde{\Sigma}_{s,t} \Big) , \ \textrm{ as } n \to \infty,
$$
\noindent where $\tilde{\Sigma}_{s,t}$ is a covariance matrix of $(\tilde{Y}_s,\tilde{Y}_t)$ for $\tilde{Y}$ as in the statement of Theorem \ref{FCLT<ptilde}.
\end{proposition}

\begin{proof}
Let
\[
p=\tilde{p}(\beta)=\frac{2\beta+1}{2(\beta+1)},
\ \text{and hence} \ 
\vartheta=2p-1=\frac{\beta}{\beta+1}.
\]
The computations leading to \eqref{phi_2} use only decomposition \eqref{decompZ}, the independence of the sequences $(Z_k^+)_{k\geq1}$ and $(Z_k^-)_{k\geq1}$ from the AERW, and the conditional independence of the centered random step sizes given the entire AERW trajectory. Therefore, \eqref{phi_2} remains valid at $p=\tilde{p}(\beta)$ and hence, the joint characteristic function of $\left(\frac{W_{\lfloor ns \rfloor} - \lfloor ns \rfloor\lambda}{\sqrt{n}},\frac{W_{\lfloor nt \rfloor} - \lfloor nt \rfloor\lambda}{\sqrt{n}}  \right)$, which we denote by $\tilde{\psi}_{n,s,t}(\theta_1,\theta_2) = \tilde{\psi}_n(\theta_1,\theta_2)$ for $(\theta_1,\theta_2) \in \mathbb{R}^2$, satisfies the following expression

\begin{align}\label{phi_3}
\lim_{n\rightarrow \infty} \tilde{\psi}_n(\theta_1,\theta_2) & = \exp\left\{ - \frac{(\theta_1 + \theta_2)^2s}{2} \left[ \frac{(\sigma^{+})^2 + (\sigma^{-})^2}{2} \right]\right\} \exp\left\{ - \frac{\theta_2^2(t-s)}{2} \left[ \frac{(\sigma^{+})^2 + (\sigma^{-})^2}{2} \right]\right\} \nonumber\\
& \times \lim_{n\rightarrow \infty} \EE \Big[ \exp \Big( i \rho \theta_1 \sum_{k=1}^{\lfloor ns \rfloor} \frac{X_k}{\sqrt{n}} + i \rho \theta_2 \sum_{k=1}^{\lfloor nt \rfloor} \frac{X_k}{\sqrt{n}} \Big) \Big].
\end{align}

Now we have to deal with the last term in \eqref{phi_3}. From \cite[Theorem 2.11]{MM} it follows that the last term in \eqref{phi_3} is equal to 
\begin{align*}
\exp\Big\{&- \frac{\rho^2}{2}\Big[ \theta_1^2s\big[ \beta^2 + (\beta+1)^2\big] + 2\theta_1 \theta_2 s \Big[ \beta^2 + (\beta+1)^2 + \beta(\beta+1)\log \Big( \frac{t}{s} \Big)\Big]\\ 
&+ \theta_2^2t\big[ \beta^2 + (\beta+1)^2\big] \Big] \Big\}.
\end{align*}

Although \cite{MM} assumes $q = 1/2$ for the Rademacher parameter that determines the first increment of the AERW, the same conclusion holds for every $q \in (0,1)$ since  the martingale functional central limit theorem used in the proof of \cite[Theorem 2.11]{MM} can be applied by conditioning on $X_1 = 1$ or $X_1 = -1$, with the same limit in both cases. Then, going back to \eqref{phi_3}, we have that
\begin{align*}
\lim_{n\rightarrow \infty} \tilde{\psi}_n(\theta_1,\theta_2) & = \exp \Big\{ - \frac{1}{2} \Big[ \theta_1^2\Big( \frac{s(\sigma^+)^2 + s(\sigma^-)^2}{2} + s \rho^2 \big[ \beta^2 + (\beta+1)^2\big]\Big)\\
&+ 2 \theta_1 \theta_2 \Big( \frac{s(\sigma^+)^2 + s(\sigma^-)^2}{2} + s\rho^2 \Big[ \beta^2 + (\beta+1)^2 + \beta(\beta+1)\log \Big( \frac{t}{s} \Big)\Big]\Big)\\
&+ \theta_2^2\Big( \frac{t(\sigma^+)^2 + t(\sigma^-)^2}{2} + t \rho^2 \big[ \beta^2 + (\beta+1)^2\big]\Big) \Big] \Big\},
\end{align*}
which concludes the proof. 
\end{proof}

\medskip

\begin{proof}[Proof of Theorem \ref{FCLT<ptilde}]
It is straightforward to obtain the convergence of the finite-dimensional distributions analogously to the two-dimensional case from Proposition \ref{prop:finitedim<tilde}. The convergence of the finite-dimensional distributions together with tightness from Proposition \ref{prop:tightness} concludes the proof.
\end{proof}

Now we will state and prove the functional central limit theorem for the GAERW in the critical case. Notice that the critical functional limit uses the exponential time scale $\lfloor n^t\rfloor$, rather than the linear scale $\lfloor nt\rfloor$.

\begin{theorem}\label{FCLT=} If $p = \frac{4\beta + 3}{4\beta + 4}$ and $\rho^2 > 0$, then
$$
\Big( \frac{W_{\lfloor n^t \rfloor}- \lfloor n^t \rfloor\lambda}{\sqrt{n^t \log(n)}} \Big)_{t\geq0}
$$
converges in distribution on $D([0,\infty))$ to $(\rho (2 \beta + 1)B_t)_{t \geq 0}$, where $B_t$ is a standard Brownian motion.
\end{theorem}

\medskip

\begin{corollary}\label{prop:criticalcase} If $p = \frac{4\beta + 3}{4\beta + 4}$ and $\rho^2 > 0$, then
$$
\frac{W_{n} - n\lambda}{\sqrt{n\log n}} \Longrightarrow \mathcal{N}(0, (\rho(2\beta+1))^2) , \ \textrm{ as } n \to \infty.
$$
\end{corollary}

\medskip

\begin{proof}[Proof of Theorem \ref{FCLT=}]

As in the proof of Theorem \ref{FCLT<}, set
$$
\hat Y^{(n)}_{t} = \frac{W_{\lfloor n^t \rfloor} - \lfloor n^t \rfloor\lambda}{\sqrt{n^t \log(n)}} , \ t\ge 0,
$$
and use decomposition \eqref{decompZ} to write
$$
\hat Y^{(n)}_{t} = \frac{\rho}{\sqrt{n^t \log(n)}} \sum_{k=1}^{\lfloor n^t \rfloor} X_k + \frac{1}{2 \sqrt{\log(n)}} \Big( \sum_{k=1}^{\lfloor n^t \rfloor} \frac{\tilde{Z}_k^+(X_k +1)}{{\sqrt{n^t}}} -  \sum_{k=1}^{\lfloor n^t \rfloor} \frac{\tilde{Z}_k^-(X_k - 1)}{{\sqrt{n^t}}} \Big),
$$
for all $t \ge 0$. From \cite[Theorem 2.6]{LL} we have that the first term above converges to a Brownian motion with diffusion coefficient $(\rho(2\beta+1))^2$. It remains to show that 
\begin{align}\label{eq:teo2.3}
\frac{1}{2\sqrt{\log(n)}}\sum_{k=1}^{\lfloor n^t \rfloor} \frac{\tilde{Z}_k^+(X_k +1)}{{\sqrt{n^t}}} -  \frac{1}{2\sqrt{\log(n)}}\sum_{k=1}^{\lfloor n^t \rfloor} \frac{\tilde{Z}_k^-(X_k - 1)}{{\sqrt{n^t}}}
\end{align}
converges to zero uniformly on compact intervals. We will prove that for the first term of \eqref{eq:teo2.3} and the argument for the second term is analogous. Let
\[
N_m^+ = \sum_{k=1}^m \mathbf{1}_{\{X_k=1\}},
\]
and denote by $\tau_j^+$ the $j$th time at which $X_k=1$. Then, if $N_m^+ = 0$, $\sum_{k=1}^m\tilde{Z}_k^+(X_k +1)$ is equal to zero and if $N_m^+ > 0$,
\[
\sum_{k=1}^{m} \tilde{Z}_k^+(X_k +1) = 2\sum_{j=1}^{N_m^+} \tilde{Z}_{\tau_j^+}
\]
Conditionally on $\mathcal{G}_\infty$, the sequence $(\tilde{Z}_{\tau_j^+})_{j \geq 1}$ is an i.i.d. sequence with mean zero and finite variance. Therefore, by the law of the iterated logarithm,
\[
\Big| \sum_{j=1}^{N_m^+} \tilde{Z}_{\tau_j^+} \Big| = \mathcal{O}\big(\sqrt{N_m^+ \log \log N_m^+}\big), \ \mbox{a.s}.
\]
Since $N_m^+ \leq m$, it follows that $\sum_{j=1}^{N_m^+} \tilde{Z}_{\tau_j^+}$ is of order at most $\sqrt{m \log \log m}$. Hence, there exists an almost surely finite random variable $C$ such that, for every $T > 0$,
\[
\sup_{0 < t \leq T} \left| \frac{1}{2\sqrt{\log(n)}}\sum_{k=1}^{\lfloor n^t \rfloor} \frac{\tilde{Z}_k^+(X_k +1)}{{\sqrt{n^t}}}\right| = \sup_{0 < t \leq T} \frac{1}{\sqrt{\log(n)}} \left| \sum_{k=1}^{N_{\lfloor n^t \rfloor}^+} \frac{\tilde{Z}_{\tau_j^+}}{{\sqrt{n^t}}} \right| \leq C \frac{\sqrt{n^t\log \log (n^t \vee e^e)}}{\sqrt{n^t \log n}},
\]
which converges to zero almost surely as $n$ goes to infinity, concluding the proof.

\end{proof}

\medskip

We next state and prove the functional convergence for the GAERW in the superdiffusive case.

\begin{theorem}\label{FCLT>} If $p > \frac{4\beta + 3}{4\beta + 4}$, then
$$
\Big( \frac{W_{\lfloor nt \rfloor} - \lfloor nt \rfloor\lambda}{n^{\vartheta(\beta+1)-\beta}} \Big)_{t \geq 0}
$$
converges in distribution on $D([0,\infty))$ to $(t^{\vartheta(\beta+1)-\beta} \rho \mathcal{L}_{q,\beta})_{t\ge 0}$, where $\mathcal{L}_{q,\beta}$ is a non-degenerate random variable. 
\end{theorem}

\medskip

\begin{corollary}\label{prop:superdiffusivecase} If $\frac{4\beta + 3}{4\beta + 4} < p < 1$ then
$$
\frac{W_{n} - n\lambda}{n^{\vartheta (\beta+1)-\beta}} \Longrightarrow \rho \mathcal{L}_{q,\beta} , \ \textrm{ as } n \to \infty,
$$
\noindent where $\mathcal{L}_{q,\beta}$ is a non-degenerate random variable. 
\end{corollary}

\medskip

\begin{remark}
Regarding the known distributional properties currently available for $\mathcal{L}_{q,\beta}$, \cite{LL} proves that $\mathcal{L}_{q,\beta}$ is square-integrable and provides explicit expressions for its first two moments. Moreover, the symmetry of the construction gives  $\mathcal{L}_{q,\beta} \overset{d}{=} (2\xi_q -1) \mathcal{L}_{1,\beta}$, where $\xi_q$ is a Bernoulli random variable with parameter $q$ independent of $\mathcal{L}_{1,\beta}$. In the particular case of the ERW ($\beta = 0$), results from \cite{GLR,GLR24} show that it is non-Gaussian and has a positive, bounded and smooth density supported on $\mathbb{R}$. In this particular case, since \(L_{q,0}\) is non-Gaussian and has finite variance, it cannot have a non-degenerate stable distribution. Moreover, its sub-Gaussian tail behavior, together with its non-Gaussianity, rules out infinite divisibility; see \cite[Corollary 9.9]{SH}.\

Whenever $\rho \neq 0$, the limiting process $(t^{\vartheta(\beta+1)-\beta} \rho \mathcal{L}_{q,\beta})_{t\ge 0}$ is not a Lévy process, since its increments are dependent multiples of the same non-degenerate random variable.
\end{remark}


\begin{proof}[Proof of Theorem \ref{FCLT>}]
Set
$$
\overline Y^{(n)}_{t} = \frac{W_{\lfloor nt \rfloor} - \lfloor nt \rfloor\lambda}{n^{\vartheta(\beta+1)-\beta}} , \ t\ge 0
$$
and use decomposition \eqref{decompZ} to write
\begin{align}\label{eq:teo2.4}
\overline Y^{(n)}_{t} = \frac{\rho}{n^{\vartheta(\beta+1)-\beta}} \sum_{k=1}^{\lfloor nt \rfloor} X_k + \frac{1}{2 n^{\vartheta(\beta+1)-\beta-\frac 12}} \Big( \sum_{k=1}^{\lfloor nt \rfloor} \frac{\tilde{Z}_k^+(X_k +1)}{{\sqrt{n}}} -  \sum_{k=1}^{\lfloor nt \rfloor} \frac{\tilde{Z}_k^-(X_k - 1)}{{\sqrt{n}}} \Big).
\end{align}
Noting that $\vartheta(\beta+1)-\beta-\frac 12 > 0$, we can follow the same steps as in the proof of Theorem \ref{FCLT=}. For the first term in
\eqref{eq:teo2.4}, we use \cite[Theorem 2.7]{LL}, with the exponent $a(\beta+1)$ printed there corrected to $a(\beta+1)-\beta$ (that is, $\vartheta(\beta+1)-\beta$ in our
notation), as required by \cite[(2.12) and (4.21)]{LL}. Hence, the
first term converges in distribution in $D([0,\infty))$ to
$(t^{\vartheta(\beta+1)-\beta} \rho \mathcal{L}_{q,\beta})_{t\ge 0}$. For the second term, a straightforward adaptation of the upper bound
used in the proof of Theorem \ref{FCLT=} shows that it converges to
zero uniformly on compact time intervals.
\end{proof}

\section{Fluctuations of the GAERW in the superdiffusive regime}\label{sec:fluctuations}

In this section, we strengthen Corollary \ref{prop:superdiffusivecase} and show that Gaussian fluctuations hold for the GAERW in the superdiffusive regime, as they do for the ERW in \cite{KT19}. 

\begin{theorem}\label{Fluctuations}
Let \((W_n)_{n\geq 1}\) be the GAERW. Assume that \(\beta\geq 0\),
\[
    \frac{4\beta+3}{4(\beta+1)}<p<1,
\]
and that \(Z_1^+\) and \(Z_1^-\) are square-integrable. Set
\[
    A=\vartheta(\beta+1),
    \qquad
    \delta=A-\beta,
\]
where \(\vartheta=2p-1\). Then there exists a non-degenerate random variable \(L_{q,\beta}\) such that
\begin{equation}\label{LLN-superdiffusive}
    \frac{W_n-n\lambda}{n^\delta}
    \xrightarrow[n \to \infty]{} \rho \mathcal{L}_{q,\beta}
    \qquad\text{a.s.}
\end{equation}
Moreover,
\begin{equation}\label{Fluc-superdiffusive}
    \frac{W_n-n\lambda-\rho \mathcal{L}_{q,\beta}n^\delta}{\sqrt n}
    \xrightarrow[n \to \infty]{\mathrm d}\mathcal N(0,\mathfrak{s}^2),
\end{equation}
where
\begin{equation}\label{eq:GAERW-fluctuation-variance}
    \mathfrak{s}^2=
    \frac{(\sigma^+)^2+(\sigma^-)^2}{2}
    +\rho^2\left(
       \frac{\beta^2}{\delta^2}
       +\frac{A^2}{\delta^2(2\delta-1)}
    \right).
\end{equation}
\end{theorem}

\medskip

\begin{proof} Let us prove \eqref{LLN-superdiffusive} first. Using \eqref{incr-form}, setting $\tilde{Z}_k^{\pm} = Z_k^{\pm} - \rho^{\pm}$, $S_n = \sum_{k=1}^n X_k$ and
$$
V_n = \sum_{k=1} ^n \left[\frac{\left(1+X_k\right)}{2} \tilde{Z}_k^+ +  \frac{\left(1-X_k\right)}{2} \tilde{Z}_k^- \right],
$$
write
\begin{equation}
    W_n = n \lambda +\rho S_n + V_n.
\end{equation}
Recall that $\mathcal{F}_k = \sigma(X_j, Z_j^+,Z_j^-: 1 \leq j \leq k)$ and observe that $(V_n)_{n\ge 0}$ is a square-integrable $(\mathcal{F}_n)$-martingale. Thus here we are using a martingale that is distinct from that used in the proof of Lemma \ref{LLN}. Write 
\begin{equation} \label{forget}
    \frac{W_n-n\lambda}{n^{\delta}} = \rho \frac{S_n}{n^{\delta}} + \frac{V_ n}{n^{\delta}}.
\end{equation}
Set $Y_n = \sum_{k=1}^n X_k \mu_k$ and $M_n = a_n Y_n$ where
\begin{equation}
a_n=\prod_{k=1}^{n-1}\gamma_k^{-1}
=\frac{\Gamma(n)\Gamma\!\bigl(A+1\bigr)}
{\Gamma\!\bigl(n+A\bigr)},
\ \, \text{with} \ \,
\gamma_n=1+\frac{A}{n}
\end{equation}
as in \cite{LL}. Recall that $\mathcal{G}_k = \sigma(X_j: 1 \leq j \leq k)$. The process $(M_n)_{n\ge 0}$ is a square-integrable $(\mathcal{G}_n)$-martingale. Based on \cite[(1.12)]{LL}, define also
\begin{equation}\label{eq:N-definition}
    N_n
    =S_n+\frac{A}{\beta-A}\mu_n^{-1}Y_n
    =S_n-\frac{A}{\delta}\mu_n^{-1}Y_n
\end{equation}
and
\[
    r_n=\frac{A}{\delta}(\mu_na_n)^{-1}.
\]
The process $(N_n)_{n\ge0}$ is also a square-integrable $(\mathcal G_n)$-martingale. Since the sequences $(Z_k^+)_{k\ge1}$ and $(Z_k^-)_{k\ge1}$ are independent of $\mathcal G_\infty$, both $M$ and $N$ remain square-integrable martingales with respect to the enlarged filtration $\mathcal F_n$. It follows from \eqref{eq:N-definition} that
\begin{equation}\label{eq:R-N-M-decomposition}
    S_n=N_n+r_nM_n.
\end{equation}

Set
\[
    C=\frac{\Gamma(A+1)}{\Gamma(\beta+1)}.
\]
We have
\begin{equation}\label{eq:an-mun-asymptotics}
    a_n\mu_n
    =
    C\frac{\Gamma(n+\beta)}{\Gamma(n+A)}
    =
    Cn^{-\delta}\bigl(1+O(n^{-1})\bigr).
\end{equation}
Consequently,
\begin{equation}\label{eq:rn-asymptotics}
    r_n
    =
    Bn^\delta\bigl(1+O(n^{-1})\bigr),
    \qquad
    B=
    \frac{A\Gamma(\beta+1)}
         {\delta\Gamma(A+1)}.
\end{equation}

Going back to \eqref{forget}, we have
\begin{equation} \label{forget2}
    \frac{W_n-n\lambda}{n^{\delta}} = \rho \frac{r_n M_n}{n^{\delta}} + \frac{V_n + \rho N_n}{n^{\delta}}.
\end{equation}
In the superdiffusive regime, the proof of Theorem 2.7 in
\cite{LL} shows that \(M_n\) converges almost surely and in
\(L^2\) to a non-degenerate random variable \(M_\infty\). Again as in \cite{LL}, define
\begin{equation}\label{eq:L-definition}
    \mathcal{L}_{q,\beta}=BM_\infty
\end{equation}
and use \eqref{eq:rn-asymptotics} to obtain that 
$\displaystyle \frac{r_n M_n}{n^{\delta}}$ converges almost surely to  $\mathcal{L}_{q,\beta}$. On the other hand, a straightforward calculation shows that the predictable quadratic variation of $V_n$ is
\begin{align*}
\langle V \rangle_n & = \sum_{k=1}^n \mathbb{E}\left[(\Delta V_k)^2 | 
    \mathcal{F}_{k-1}\right] =\sum_{k=1}^n \left( \mathbb{E}\left[\frac{1+X_k}{2}\Big|\mathcal{F}_{k-1}\right] (\sigma^+)^2 + \mathbb{E}\left[\frac{1-X_k}{2}\Big|\mathcal{F}_{k-1}\right] (\sigma^-)^2 \right) \\
& = q(\sigma^+)^2 + (1-q)(\sigma^-)^2 + \frac{(n-1) ((\sigma^+)^2 + (\sigma^-)^2)}{2} \\
& \qquad \qquad \qquad \qquad \qquad \qquad \qquad + \frac{((\sigma^+)^2 - (\sigma^-)^2)}{2} \vartheta (\beta+1) \sum_{k=2}^n \overline X_{k-1}^\mu
\end{align*}
which is asymptotically linear, since by \cite[(4.5)]{LL} $\overline X_{k-1}^\mu \rightarrow 0$ a.s. Moreover, \cite[(3.5)]{LL} gives 
\[
    \lim_{n\to\infty}\frac{\langle N\rangle_n}{n}
    =\frac{\beta^2}{\delta^2}
\]
(note that $N_n\equiv0$ if $\beta=0$). As verified below in
\eqref{eq:UN-cross-moment}, we also have $\langle V,N\rangle_n=0$. Therefore,
$(V_n+\rho N_n)_{n\geq0}$ is a square-integrable
$(\mathcal F_n)$-martingale such that
\begin{equation}
\label{lemma:QV-N}
\lim_{n\to\infty}\frac{\langle V+\rho N\rangle_n}{n}
=
\lim_{n\to\infty}
\frac{\langle V\rangle_n+\rho^2\langle N\rangle_n}{n}
=
\frac{(\sigma^+)^2+(\sigma^-)^2}{2}
+\left(\frac{\rho\beta}{\delta}\right)^2.
\end{equation}
We indeed show a stronger result in the proof of \eqref{Fluc-superdiffusive} below. 
 
Since
$p > \frac{4\beta+3}{4\beta+4}$, we have that
$\delta=\vartheta(\beta+1)-\beta>1/2$. The increments of $(V_n+\rho N_n)_{n\geq 0}$ are
square-integrable and have uniformly bounded second moments. Consequently, 
\[
\sum_{k=1}^{\infty}
\frac{
    \big(\Delta V_k+\rho\Delta N_k\big)^2
}{k^{2\delta}}
<\infty
\qquad\mbox{a.s.}
\]
Furthermore, by the orthogonality of martingale increments,
\[
\sup_{n\geq 1}
\EE\left[
    \left(
        \sum_{k=1}^{n}
        \frac{\Delta V_k+\rho\Delta N_k}{k^\delta}
    \right)^2
\right]
=
\sum_{k=1}^{\infty}
\frac{
    \EE\left[
        \big(\Delta V_k+\rho\Delta N_k\big)^2
    \right]
}{k^{2\delta}}
<\infty.
\]
Thus,
\[
\left(
    \sum_{k=1}^{n}
    \frac{\Delta V_k+\rho\Delta N_k}{k^\delta}
\right)_{n\geq 1}
\]
is a square-integrable martingale bounded in $L^2$, and hence it
converges almost surely. Therefore, by Kronecker's lemma
(see \cite[p.~31]{HH80}),
\[
\frac{V_n+\rho N_n}{n^\delta}
\xrightarrow[n \to \infty]{} 0
\qquad\mbox{a.s.}
\]
Therefore, \eqref{LLN-superdiffusive} holds.

\medskip

We now prove \eqref{Fluc-superdiffusive}. For \(k\geq1\), write the
martingale increments of \(V\) as
\begin{equation}\label{eq:U-definition}
    U_k=\Delta V_k
    =
    \frac{1+X_k}{2}\widetilde Z_k^+
    +\frac{1-X_k}{2}\widetilde Z_k^-.
\end{equation}
Following \cite{LL}, for \(k\geq2\), let
\[
    \varepsilon_k
    =
    Y_k-\left(1+\frac{A}{k-1}\right)Y_{k-1}.
\]
This is exactly the martingale innovation used in \cite{LL}, with the
identification \(A=\vartheta(\beta+1)\) and with \(k=n+1\) in the notation
\(\varepsilon_{n+1}=Y_{n+1}-\gamma_nY_n\) of that paper.
If
\[
    h_{k-1}=\mathbb E[X_k\mid\mathcal G_{k-1}]=\mathbb E[X_k\mid\mathcal F_{k-1}],
\]
then
\begin{equation}\label{eq:epsilon-expression}
    \varepsilon_k=\mu_k(X_k-h_{k-1}).
\end{equation}
Therefore,
\begin{equation}\label{eq:epsilon-properties}
    \mathbb E[\varepsilon_k^2\mid\mathcal G_{k-1}]
    =\mu_k^2(1-h_{k-1}^2),
    \qquad
    |\varepsilon_k|\leq2\mu_k.
\end{equation}
Moreover,
\begin{equation}\label{eq:martingale-increments}
    \Delta M_k=a_k\varepsilon_k,
    \qquad
    \Delta N_k
    =-\frac{\beta}{\delta}\frac{\varepsilon_k}{\mu_k}.
\end{equation}

By equation~(4.5) of \cite{LL},
\[
    \frac{Y_n}{n\mu_n}\xrightarrow[n \to \infty]{} 0
    \qquad\text{a.s.}
\]
Since
\[
    h_{k-1}
    =\frac{A}{(k-1)\mu_k}Y_{k-1}
\]
and \(\mu_k/\mu_{k-1}\to1\), it follows that
\begin{equation}\label{eq:h-to-zero}
    h_k\xrightarrow[k \to \infty]{} 0
    \qquad\text{a.s.}
\end{equation}

Using
\[
    M_\infty-M_n
    =\sum_{k=n+1}^{\infty}a_k\varepsilon_k,
\]
we can write
\begin{align}
    &\frac{W_n-n\lambda-\rho L_{q,\beta}n^\delta}{\sqrt n}
    \notag\\
    &\quad=
    \frac{V_n+\rho N_n+\rho r_n(M_n-M_\infty)}{\sqrt n}
    +
    \rho\frac{(r_n-Bn^\delta)M_\infty}{\sqrt n}.
    \label{eq:central-decomposition}
\end{align}
By \eqref{eq:rn-asymptotics},
\[
    r_n-Bn^\delta=O(n^{\delta-1}).
\]
Since \(\delta<1\), the last term in
\eqref{eq:central-decomposition} converges almost surely to zero.

We apply the martingale triangular-array central limit theorem to the first
term on the right-hand side of \eqref{eq:central-decomposition}, see \cite[Corollary 3.1]{HH80}. To obtain a
finite array, set \(m_n=n^2\) and define, for \(2\leq k\leq m_n\),
\[
    D_{n,k}
    =
    \begin{cases}
    \displaystyle
    \frac{1}{\sqrt n}
    \left(
       U_k-\rho\frac{\beta}{\delta}
       \frac{\varepsilon_k}{\mu_k}
    \right),
       &2\leq k\leq n,\\[3ex]
    \displaystyle
    -\frac{\rho r_n}{\sqrt n}a_k\varepsilon_k,
       &n<k\leq m_n.
    \end{cases}
\]
By independence of the sequences \((Z_j^+)\) and \((Z_j^-)\) from the
AERW, all the conditional identities for \(\varepsilon_k\) remain valid with
\(\mathcal F_{k-1}\) in place of \(\mathcal G_{k-1}\). Hence
\((D_{n,k})\) is a martingale-difference array with respect to
\((\mathcal F_k)\). The finitely many initial terms omitted from the array are
\(o_{\mathbb P}(1)\) after division by \(\sqrt n\).

We next explain why no cross-covariance terms are missing from the limiting
variance. For \(k\leq n\),
\[
    \Delta V_k=U_k,
    \qquad
    \Delta N_k
    =-\frac{\beta}{\delta}\frac{\varepsilon_k}{\mu_k}.
\]
Conditionally on \(\mathcal F_{k-1}\) and \(X_k\), the variable \(U_k\) is
either \(\widetilde Z_k^+\) or \(\widetilde Z_k^-\). Both variables are
centered and independent of the AERW. Consequently,
\begin{equation}\label{eq:UN-cross-moment}
    \mathbb E[U_k\varepsilon_k\mid\mathcal F_{k-1}]=0,
\end{equation}
and therefore
\[
    \langle V,N\rangle_n
    =
    \sum_{k=2}^n
    \mathbb E[\Delta V_k\Delta N_k\mid\mathcal F_{k-1}]
    =0.
\]

The cross terms involving \(M_n-M_\infty\) vanish for a different reason.
Since \(M_n\to M_\infty\) in \(L^2\), and \(M\) remains a martingale with
respect to the enlarged filtration by independence of the random step sizes,
\[
    M_n=\mathbb E[M_\infty\mid\mathcal F_n].
\]
Since \(V_n\) and \(N_n\) are \(\mathcal F_n\)-measurable, it follows that
\begin{equation}\label{eq:past-future-cross-moments}
    \mathbb E[V_n(M_n-M_\infty)]
    =
    \mathbb E[N_n(M_n-M_\infty)]
    =0.
\end{equation}
These identities express the separation between the increments up to time
\(n\), which form \(V_n\) and \(N_n\), and the future increments
\[
    M_n-M_\infty
    =-\sum_{k=n+1}^{\infty}\Delta M_k.
\]
Notice that \(\langle N,M\rangle_n\) is not zero in general. Indeed,
\[
    \langle N,M\rangle_n
    =
    -\frac{\beta}{\delta}
    \sum_{k=2}^n a_k\mu_k(1-h_{k-1}^2).
\]
It does not contribute here because the \(N\)-component of the triangular
array uses increments with \(k\leq n\), whereas its \(M\)-component uses
increments with \(k>n\). Thus, within a fixed row of the array, the only two
components occurring at the same time are \(U_k\) and \(\Delta N_k\), and
their conditional cross moment is zero by
\eqref{eq:UN-cross-moment}.

We now record explicitly the regular-variation result that will be used twice
below. Define the positive function
\begin{equation}\label{eq:q-function-definition}
    g(x)
    =
    \left(
       C\frac{\Gamma(x+\beta)}{\Gamma(x+A)}
    \right)^2,
    \qquad
    C=\frac{\Gamma(A+1)}{\Gamma(\beta+1)},
    \qquad x\geq1.
\end{equation}
In particular,
\[
    g(k)=(a_k\mu_k)^2.
\]
The standard asymptotic formula for a quotient of Gamma functions gives
\begin{equation}\label{eq:q-regular-variation}
    g(x)\sim C^2x^{-2\delta}.
\end{equation}
Thus \(g\) is regularly varying at infinity with index \(-2\delta\); explicitly,
\begin{equation}\label{Potter}
    \frac{g(tx)}{g(x)}\xrightarrow[x \to \infty]{} t^{-2\delta}
    \qquad\text{for every }t>0.
\end{equation}
Because \(\delta>1/2\), this index is strictly smaller than \(-1\). The
discrete tail form of Karamata's theorem states that, if \(f\) is positive and
regularly varying with index \(\alpha<-1\), then
\[
    \frac{\sum_{k=n+1}^{\infty}f(k)}{nf(n)}
    \xrightarrow[n \to \infty]{} \frac{1}{-\alpha-1}.
\]
See \cite[Theorem~1.5.11 and Section~1.9]{BGT}. The discrete statement follows
by applying the integral form of Karamata's theorem to a step-function
interpolation of the sequence. Applying this result to the function \(g\) in
\eqref{eq:q-function-definition}, whose index is \(\alpha=-2\delta\), yields
\begin{equation}\label{eq:Karamata-q-tail}
    \frac{\sum_{k=n+1}^{\infty}(a_k\mu_k)^2}
         {n(a_n\mu_n)^2}
    =
    \frac{\sum_{k=n+1}^{\infty}g(k)}{ng(n)}
    \xrightarrow[n \to \infty]{} \frac{1}{2\delta-1}.
\end{equation}

We first use \eqref{eq:Karamata-q-tail} to show that the part of the martingale
tail after \(m_n\) is negligible. Since
\[
    r_n=\frac{A}{\delta}(a_n\mu_n)^{-1}
        =\frac{A}{\delta}g(n)^{-1/2},
\]
Karamata's limit \eqref{eq:Karamata-q-tail}, now applied with \(m_n\) in place
of \(n\), gives
\begin{align*}
    \frac{r_n^2}{n}\sum_{k>m_n}(a_k\mu_k)^2
    &=
    \frac{A^2}{\delta^2}
    \frac{\sum_{k>m_n}g(k)}{ng(n)}\\
    &\sim
    \frac{A^2}{\delta^2(2\delta-1)}
    \frac{m_ng(m_n)}{ng(n)}.
\end{align*}
By the explicit asymptotic relation \eqref{eq:q-regular-variation},
\[
    \frac{m_ng(m_n)}{ng(n)}
    \sim
    \left(\frac{m_n}{n}\right)^{1-2\delta}
    =n^{1-2\delta} \xrightarrow[n \to \infty]{} 0.
\]
Consequently,
\begin{equation}\label{eq:truncated-tail-negligible}
    \frac{r_n^2}{n}\sum_{k>m_n}(a_k\mu_k)^2
    \xrightarrow[n \to \infty]{} 0.
\end{equation}
We now justify precisely the orthogonality argument used to control the
truncated martingale tail. Since
\[
    \mathbb E[\varepsilon_k\mid\mathcal G_{k-1}]=0,
\]
the variables \(a_k\varepsilon_k=\Delta M_k\) are square-integrable
martingale differences. Therefore, if \(m<\ell\),
$$
    \mathbb E\Big[
       \Big(\sum_{k=m+1}^{\ell}a_k\varepsilon_k\Big)^2
    \Big]
    =
    \sum_{k=m+1}^{\ell}a_k^2\mathbb E[\varepsilon_k^2].
$$
Since \(M_n\to M_\infty\) in \(L^2\), letting \(\ell\to\infty\) gives
$$
    \mathbb E\Big[ \Big(\sum_{k=m+1}^{\infty}a_k\varepsilon_k\Big)^2
    \Big]
    =
    \sum_{k=m+1}^{\infty}a_k^2\mathbb E[\varepsilon_k^2].
$$
Moreover, by \eqref{eq:epsilon-properties},
\[
    \mathbb E[\varepsilon_k^2]
    =\mathbb E[\mu_k^2(1-h_{k-1}^2)]
    \leq\mu_k^2.
\]
Consequently,
\begin{align*}
    \mathbb E\Big[
       \Big|
       \frac{\rho r_n}{\sqrt n}
       \sum_{k>m_n}a_k\varepsilon_k
       \Big|^2
    \Big]
    &=
    \frac{\rho^2r_n^2}{n}
    \sum_{k>m_n}a_k^2\mathbb E[\varepsilon_k^2]\\
    &\leq
    \frac{\rho^2r_n^2}{n}
    \sum_{k>m_n}(a_k\mu_k)^2
    \longrightarrow0
\end{align*}
by \eqref{eq:truncated-tail-negligible}. Hence
\[
    \frac{\rho r_n}{\sqrt n}
    \sum_{k>m_n}a_k\varepsilon_k
    \xrightarrow[n \to \infty]{} 0
    \qquad\text{in }L^2.
\]

We next compute the limiting conditional variance. The limits established
immediately before this proof give
\[
    \frac{\langle V\rangle_n}{n}
    \xrightarrow[n \to \infty]{} \frac{(\sigma^+)^2+(\sigma^-)^2}{2}
    \qquad\text{and}\qquad
    \frac{\langle N\rangle_n}{n}
    \xrightarrow[n \to \infty]{} \frac{\beta^2}{\delta^2}
    \qquad\text{a.s.}
\]
Together with \(\langle V,N\rangle_n=0\), proved above, these limits imply
\begin{align}
    \frac{1}{n}\langle V+\rho N\rangle_n
    &=
    \frac{\langle V\rangle_n}{n}
    +\frac{2\rho\langle V,N\rangle_n}{n}
    +\rho^2\frac{\langle N\rangle_n}{n}
    \notag\\
    &\xrightarrow[n \to \infty]{}
    \frac{(\sigma^+)^2+(\sigma^-)^2}{2}
    +\rho^2\frac{\beta^2}{\delta^2}
    \qquad\text{a.s.}
    \label{eq:past-variance}
\end{align}

For the contribution from \(k>n\), observe from
\eqref{eq:epsilon-properties} that
\[
    a_k^2\mathbb E[\varepsilon_k^2\mid\mathcal G_{k-1}]
    =g(k)(1-h_{k-1}^2).
\]
Furthermore, \eqref{eq:h-to-zero} implies
\[
    \sup_{k>n}|h_{k-1}| \xrightarrow[n \to \infty]{} 0
    \qquad\text{a.s.}
\]
Hence the factors \(1-h_{k-1}^2\) do not change the asymptotic value of the
tail sum. Using again \eqref{eq:Karamata-q-tail}, namely Karamata's theorem for
the function \(g\) of index \(-2\delta\), we obtain
\begin{align}
    &\frac{\rho^2r_n^2}{n}
    \sum_{k=n+1}^{\infty}
    a_k^2\mathbb E[\varepsilon_k^2\mid\mathcal F_{k-1}]
    \notag\\
    &\quad=
    \rho^2\frac{A^2}{\delta^2}
    \frac{\sum_{k=n+1}^{\infty}g(k)(1-h_{k-1}^2)}{ng(n)}
    \notag\\
    &\quad \xrightarrow[n \to \infty]{}
    \rho^2\frac{A^2}{\delta^2(2\delta-1)}
    \qquad\text{a.s.}
    \label{eq:future-variance}
\end{align}
The same limit holds with the sum truncated at \(m_n\), by
\eqref{eq:truncated-tail-negligible}. Combining
\eqref{eq:past-variance} and \eqref{eq:future-variance}, the conditional
variance of the martingale array converges almost surely to
\[
    \frac{(\sigma^+)^2+(\sigma^-)^2}{2}
    +\rho^2\frac{\beta^2}{\delta^2}
    +\rho^2\frac{A^2}{\delta^2(2\delta-1)}
    = \mathfrak{s}^2.
\]

It remains to verify the conditional Lindeberg condition. For $k\leq n$, set
\[
c_k=\frac{\rho\beta\varepsilon_k}{\delta\mu_k}.
\]
By \eqref{eq:epsilon-properties},
\[
|c_k|\leq C:=\frac{2|\rho|\beta}{\delta}.
\]
Hence, for every $\eta>0$ and all sufficiently large $n$,
\[
|U_k-c_k|>\eta\sqrt{n}
\quad\Longrightarrow\quad
|U_k|>\frac{\eta\sqrt{n}}{2}.
\]
Moreover, on this event we have $|c_k|\leq |U_k|$, and therefore
\[
(U_k-c_k)^2\leq 4U_k^2.
\]
The contribution of $U_k$ satisfies Lindeberg's condition because
$Z_1^+$ and $Z_1^-$ are square-integrable. Indeed,
\[
\frac{1}{n}\sum_{k=2}^{n}
\mathbb{E}\left[
U_k^2\mathbf{1}_{\{|U_k|>\eta\sqrt{n}/2\}}
\right]
\leq
\mathbb{E}\left[
(\widetilde Z_1^+)^2
\mathbf{1}_{\{|\widetilde Z_1^+|>\eta\sqrt{n}/2\}}
\right]
+
\mathbb{E}\left[
(\widetilde Z_1^-)^2
\mathbf{1}_{\{|\widetilde Z_1^-|>\eta\sqrt{n}/2\}}
\right]
\longrightarrow 0.
\]
If $L_n$ denotes the corresponding conditional Lindeberg sum, then
\[
L_n
:=
\frac{1}{n}\sum_{k=2}^{n}
\mathbb{E}\left[
(U_k-c_k)^2
\mathbf{1}_{\{|U_k-c_k|>\eta\sqrt{n}\}}
\,\middle|\,\mathcal{F}_{k-1}
\right]
\leq
\frac{4}{n}\sum_{k=2}^{n}
\mathbb{E}\left[
U_k^2
\mathbf{1}_{\{|U_k|>\eta\sqrt{n}/2\}}
\,\middle|\,\mathcal{F}_{k-1}
\right].
\]
The preceding estimate and the tower property yield
$\mathbb{E}[L_n]\to0$. Since $L_n\geq0$, Markov's inequality implies
that $L_n\to0$ in probability.

For \(k>n\), the sequence \((a_k\mu_k)\) is decreasing, since
\[
    \frac{a_{k+1}\mu_{k+1}}{a_k\mu_k}
    =\frac{k+\beta}{k+A}<1.
\]
Therefore,
\[
    \sup_{k>n}
    \frac{r_na_k|\varepsilon_k|}{\sqrt n}
    \leq
    \frac{2A}{\delta\sqrt n}
    \sup_{k>n}
    \frac{a_k\mu_k}{a_n\mu_n}
    \leq
    \frac{2A}{\delta\sqrt n}
    \xrightarrow[n \to \infty]{} 0.
\]
Thus, the conditional Lindeberg condition holds for the entire array.

The martingale triangular-array central limit theorem now yields
\[
    \frac{
      V_n+\rho N_n+\rho r_n(M_n-M_\infty)
    }{\sqrt n}
    \xrightarrow[n \to \infty]{\mathrm d}\mathcal N(0,s^2).
\]
Together with \eqref{eq:central-decomposition}, this proves
\eqref{Fluc-superdiffusive}.
\end{proof}

\section{Limit theorems for the GERW with infinite-variance increments}\label{sec:stable}

We now assume that the distributions of the sequences $(Z_n^+)_{n\geq 1}$
and $(Z_n^-)_{n\geq 1}$ are in the domains of attraction of stable laws with
parameters $\alpha^+,\alpha^-\in(0,2)$, respectively. More precisely, there
exist non-degenerate stable random variables $S_{\alpha^+}^+$ and
$S_{\alpha^-}^-$, normalizing sequences $a_n^+,a_n^->0$, and centering
constants $b_n^+,b_n^-\in\mathbb{R}$ such that
\[
 \frac{1}{a_n^+}\left(\sum_{k=1}^n Z_k^+-b_n^+\right)
 \xrightarrow{d}S_{\alpha^+}^+
 \quad\text{and}\quad
 \frac{1}{a_n^-}\left(\sum_{k=1}^n Z_k^--b_n^-\right)
 \xrightarrow{d}S_{\alpha^-}^-
 \quad\text{as }n\to\infty.
\]
The normalizing sequences may be chosen so that
$a_n^\pm\sim n^{1/\alpha^\pm}L^\pm(n)$, where $L^+$ and $L^-$ are slowly
varying. With $F_\pm$ denoting the distributions of $Z_1^\pm$, the standard
centering convention is
\[
b_n^\pm=
\begin{cases}
 n\rho^\pm, & \alpha^\pm>1,\\[0.3em]
 n\displaystyle\int_{|x|\leq a_n^\pm}x\,dF_\pm(x)
 =n\EE\!\left[Z_1^\pm\mathds{1}_{\{|Z_1^\pm|\leq a_n^\pm\}}\right],
 & \alpha^\pm=1,\\[0.8em]
 0, & \alpha^\pm<1,
\end{cases}
\]
where $\rho^\pm=\EE[Z_1^\pm]$ when $\alpha^\pm>1$, see
\cite[Chapter XVII]{Feller}.

Recall that an $\alpha$-stable random variable
$S_\alpha(\gamma,\zeta,\delta)$, with scale parameter $\gamma>0$, skewness
$\zeta\in[-1,1]$, and location parameter $\delta\in\mathbb{R}$, has
characteristic function
\[
\varphi(u)=
\begin{cases}
\exp\!\left\{i\delta u-\gamma^\alpha|u|^\alpha
\left(1-i\zeta\,\operatorname{sign}(u)
\tan\dfrac{\pi\alpha}{2}\right)\right\}, & \alpha\neq1,\\[1.1em]
\exp\!\left\{i\delta u-\gamma|u|
\left(1+i\zeta\,\dfrac{2}{\pi}\operatorname{sign}(u)
\log|u|\right)\right\}, & \alpha=1.
\end{cases}
\]
Its L\'evy--Khintchine representation, with truncation
$x\mathds{1}_{\{|x|\leq1\}}$, is
\[
 \varphi(u)=\exp\!\left\{i\eta u+
 \int_{\mathbb{R}}\left(e^{iux}-1-iux\mathds{1}_{\{|x|\leq1\}}\right)
 \nu(dx)\right\}.
\]
Here $\eta$ is the drift associated with this truncation and should not be
confused with the location parameter $\delta$. The L\'evy measure has density
\[
 \nu(dx)=\left(c_+x^{-\alpha-1}\mathds{1}_{(0,\infty)}(x)
 +c_-|x|^{-\alpha-1}\mathds{1}_{(-\infty,0)}(x)\right)dx,
\]
where
\[
 c_+=C_\alpha\gamma^\alpha\frac{1+\zeta}{2},\qquad
 c_-=C_\alpha\gamma^\alpha\frac{1-\zeta}{2},
\]
and
\[
 C_\alpha=\frac{\alpha}
 {\Gamma(1-\alpha)\cos(\pi\alpha/2)},\quad \alpha\neq1,
 \qquad C_1=\frac{2}{\pi}.
\]

We shall consider the case
\[
 \alpha:=\alpha^+=\alpha^-\in(0,2)
\]
and assume that the two domains of attraction admit a common compatible
normalization
\[
 a_n:=a_n^+=a_n^-\sim n^{1/\alpha}L(n).
\]
Throughout this section, $p\in(0,1)$. As in the preceding sections, the
sequences $(Z_k^+)_{k\geq1}$ and $(Z_k^-)_{k\geq1}$ are mutually independent
and independent of the ERW.
Multiplicative constants in the two original normalizations can be absorbed
into the scale parameters of the limiting stable laws. Write
\[
 S_\alpha^+=S_\alpha(\gamma_+,\zeta_+,\delta_+),\qquad
 S_\alpha^-=S_\alpha(\gamma_-,\zeta_-,\delta_-),
\]
where the location parameters correspond to the centering convention above.
In particular, $\delta_+=\delta_-=0$ when $1<\alpha<2$.

For $t\geq0$, define the partial-sum processes
\[
 A_n^\pm(t):=\frac{1}{a_n}\left(
 \sum_{k=1}^{\lfloor nt\rfloor}Z_k^\pm
 -\frac{\lfloor nt\rfloor}{n}b_n^\pm\right).
\]
The classical stable functional limit theorem gives
\begin{equation}\label{stable-fclt-components}
 \big(A_n^+,A_n^-\big)
 \Rightarrow\big(\mathcal{S}_\alpha^+,\mathcal{S}_\alpha^-\big)
\end{equation}
in $D([0,\infty))^2$ endowed with the product $J_1$ topology, where the
limiting processes are independent stable L\'evy processes such that
$\mathcal{S}_\alpha^\pm(1)$ has the same distribution as $S_\alpha^\pm$,
see \cite[Theorem 2.7]{Sk}. Independence follows from the standing assumption
that the sequences $(Z_k^+)_{k\geq1}$ and $(Z_k^-)_{k\geq1}$ are independent.

Let $\Psi_+$ and $\Psi_-$ denote the characteristic exponents of
$S_\alpha^+$ and $S_\alpha^-$, respectively, and define
\[
 \Psi(u):=\frac{\Psi_+(u)+\Psi_-(u)}{2}.
\]
The corresponding stable law will be denoted by
$S_\alpha(\gamma,\zeta,\delta)$, where
\[
 \gamma=\left(\frac{\gamma_+^\alpha+\gamma_-^\alpha}{2}
 \right)^{1/\alpha},\qquad
 \zeta=\frac{\gamma_+^\alpha\zeta_++\gamma_-^\alpha\zeta_-}
 {\gamma_+^\alpha+\gamma_-^\alpha},\qquad
 \delta=\frac{\delta_++\delta_-}{2}.
\]
Equivalently, this is the distribution of
$\mathcal{S}_\alpha^+(1/2)+\mathcal{S}_\alpha^-(1/2)$, and its L\'evy measure
is $(\nu^++\nu^-)/2$.

Write again
\begin{equation}\label{stable-decomposition}
 \sum_{k=1}^m\mathcal{Z}_k
 =\sum_{k=1}^m\frac{1+X_k}{2}Z_k^+
 +\sum_{k=1}^m\frac{1-X_k}{2}Z_k^-.
\end{equation}
Set
\[
 S_m:=\sum_{k=1}^mX_k,\qquad
 N_m^+:=\sum_{k=1}^m\frac{1+X_k}{2}=\frac{m+S_m}{2},\qquad
 N_m^-:=\frac{m-S_m}{2}.
\]
The strong laws for the ERW give $S_m/m\to0$ almost surely, see
\cite[Theorems 3.1, 3.4 and 3.7]{B18}. Consequently, for every $T>0$,
\begin{equation}\label{clock-convergence}
 \sup_{0\leq t\leq T}\left|
 \frac{N_{\lfloor nt\rfloor}^\pm}{n}-\frac{t}{2}\right|
 \longrightarrow0\qquad\text{a.s.}
\end{equation}
Indeed, the left-hand side is bounded by
\[
 \frac{1}{n}+\frac{1}{2n}\max_{m\leq\lfloor nT\rfloor}|S_m|,
\]
which converges to zero almost surely. Recall that $\mathcal{G}_\infty=\sigma(X_1,X_2,\ldots)$ and define
\[
 b_{X,n}(t):=\frac{N_{\lfloor nt\rfloor}^+}{n}b_n^+
 +\frac{N_{\lfloor nt\rfloor}^-}{n}b_n^-.
\]
Conditionally on $\mathcal{G}_\infty$, the variables selected at the times at
which $X_k=1$ are still i.i.d. with the distribution of $Z_1^+$, and the
analogous statement holds for the minus-labeled increments. Thus, by
\eqref{stable-fclt-components}, \eqref{clock-convergence}, and the random
time-change theorem in the $J_1$ topology,
\begin{equation}\label{conditional-functional-limit}
 \left(\frac{\sum_{k=1}^{\lfloor nt\rfloor}\mathcal{Z}_k-b_{X,n}(t)}
 {a_n}\right)_{t\geq0}
 \Rightarrow
 \big(\mathcal{S}_\alpha^+(t/2)+\mathcal{S}_\alpha^-(t/2)\big)_{t\geq0}
\end{equation}
conditionally on $\mathcal{G}_\infty$, for almost every realization of the
ERW. We use here the continuity of composition at a continuous, strictly
increasing limiting clock. Addition is also continuous at the limiting pair
because two independent L\'evy processes have no common discontinuities almost
surely, see \cite[Theorem 12.7.3 and Section 13.2]{W02}.

In particular, with
\[
 b_{X,n}:=b_{X,n}(1),\qquad b_n:=\frac{b_n^++b_n^-}{2},\qquad
 \rho_n:=\frac{b_n^+-b_n^-}{2n},
\]
we have
\begin{equation}\label{conv-cond-cent}
 \frac{\sum_{k=1}^n\mathcal{Z}_k-b_{X,n}}{a_n}
 \Longrightarrow S_\alpha(\gamma,\zeta,\delta)
\end{equation}
conditionally on $\mathcal{G}_\infty$, almost surely. Moreover,
\begin{equation}\label{cond-uncond-cent}
 \frac{b_{X,n}-b_n}{a_n}=\frac{\rho_n}{a_n}S_n.
\end{equation}
If $1<\alpha<2$, then $b_n=n\lambda$ and $\rho_n=\rho$, where
\[
 \lambda=\frac{\rho^++\rho^-}{2},\qquad
 \rho=\frac{\rho^+-\rho^-}{2}.
\]
If $0<\alpha<1$, then $b_n=\rho_n=0$. For $\alpha=1$, formula
\eqref{cond-uncond-cent} remains valid with the truncated centerings given
above.

We can now state the distributional limit theorem. As before $\mathcal{L}_{q,0}$ denotes
the non-degenerate limit of the superdiffusive ERW, that is,
\[
\lim_{n \to \infty} \frac{S_n}{n^{2p-1}} = \mathcal{L}_{q,0}\qquad\text{a.s.}
\]
for $p>3/4$.

\begin{theorem}\label{alphafindimconv}
Suppose that $0<\alpha<2$ and that the two domains of attraction have the
common compatible normalization $a_n\sim n^{1/\alpha}L(n)$ specified above.
If
\[
 p<\frac{\alpha+1}{2\alpha},
\]
then
\begin{equation}\label{stablelimit1}
 \frac{\sum_{k=1}^n\mathcal{Z}_k-b_n}{a_n}
 \Longrightarrow S_\alpha(\gamma,\zeta,\delta).
\end{equation}

Assume now that $1<\alpha<2$. If $p=(\alpha+1)/(2\alpha)$ and $L(n)\to\ell$ for some
$0\leq\ell\leq\infty$, then the following alternatives hold:
\begin{enumerate}
 \item if $0<\ell<\infty$, then
 \begin{equation}\label{stablelimit2}
  \frac{\sum_{k=1}^n\mathcal{Z}_k-b_n}{a_n}
  \Longrightarrow S_\alpha(\gamma,\zeta,\delta)
  +\frac{\rho}{\ell} \mathcal{L}_{q,0},
 \end{equation}
 where the two random variables on the right-hand side are independent;
 \item if $\ell=\infty$, or if $\rho=0$, the limit in
 \eqref{stablelimit1} remains valid;
 \item if $\ell=0$ and $\rho\neq0$, the sequence on the left-hand side of
 \eqref{stablelimit2} is not tight.
\end{enumerate}

Finally, if $p>(\alpha+1)/(2\alpha)$, then, when $\rho\neq0$,
\begin{equation}\label{stablelimit3}
 \left(\frac{\sum_{k=1}^n\mathcal{Z}_k-b_n}{a_n}\right)_{n\geq1}
 \quad\text{is not tight},
\end{equation}
whereas
\begin{equation}\label{stablelimit4}
 \frac{\sum_{k=1}^n\mathcal{Z}_k-b_n}{n^{2p-1}}
 \Longrightarrow \rho \mathcal{L}_{q,0}.
\end{equation}
If $\rho=0$, the limit in \eqref{stablelimit1} remains valid for the
normalization $a_n$, while the left-hand side of \eqref{stablelimit4}
converges to zero in probability.
\end{theorem}

\begin{proof}
The conditionally centered convergence is given by
\eqref{conv-cond-cent}. It therefore remains to study the term in
\eqref{cond-uncond-cent}.

For $0<\alpha<1$, this term is identically zero. When $1<\alpha<2$, it is
$\rho S_n/a_n$. The almost sure ERW estimates in
\cite[Theorems 3.2, 3.5 and 3.7]{B18}, together
with the standard Potter bounds for the slowly varying function $L$
\cite[Theorem 1.5.6]{BGT}, show that
\[
 \frac{S_n}{a_n}\xrightarrow{\PP}0
 \qquad\text{if}\qquad
 2p-1<\frac{1}{\alpha}.
\]
This also covers the diffusive and ERW-critical regimes, since
$1/2<1/\alpha$ for $\alpha<2$.

For $\alpha=1$, write
\[
 \frac{\rho_nS_n}{a_n}
 =\frac{S_n}{n}\frac{b_n^+-b_n^-}{2a_n}.
\]
The absolute values of the truncated first moments occurring in $b_n^\pm$
are bounded by slowly varying functions. Consequently,
$(b_n^+-b_n^-)/a_n$ is of at most slowly varying order, see
\cite[Chapter XVII]{Feller}. Since $p<1$, the ERW limit theorems imply that
there is a $\kappa>0$ such that $S_n/n=O(n^{-\kappa})$ almost surely. Using again Potter bounds
\cite[Theorem 1.5.6]{BGT} we obtain that the last display converges to
zero in probability. This proves \eqref{stablelimit1} for the full range
$0<\alpha<2$.

Suppose next that $1<\alpha<2$ and
$p=(\alpha+1)/(2\alpha)$. Then $2p-1=1/\alpha$, and
\[
 \frac{\rho S_n}{a_n}
 =\rho\frac{S_n}{n^{1/\alpha}}\frac{n^{1/\alpha}}{a_n}.
\]
The almost sure convergence of the ERW gives the three alternatives in the
critical case. The conditional convergence in \eqref{conv-cond-cent} has a
non-random limiting law. Conditioning and dominated convergence therefore
also give joint convergence with the $\mathcal{G}_\infty$-measurable ERW
term, and the two limits are independent. If $\ell=0$ and $\rho\neq0$,
non-tightness follows from the non-degeneracy of $\mathcal{L}_{q,0}$.

If $p>(\alpha+1)/(2\alpha)$, then
$a_n/n^{2p-1}\to0$ by Potter bounds. Hence the conditionally centered term
in \eqref{conv-cond-cent}, divided by $n^{2p-1}$, converges to zero in
probability, while
\[
 \frac{b_{X,n}-b_n}{n^{2p-1}}
 =\rho\frac{S_n}{n^{2p-1}}
 \longrightarrow\rho \mathcal{L}_{q,0}\qquad\text{a.s.}
\]
This proves \eqref{stablelimit3}--\eqref{stablelimit4}; when $\rho=0$, the
ERW term vanishes identically and \eqref{conv-cond-cent} yields the claimed
stable limit.
\end{proof}


For $1<\alpha<2$, Theorem \ref{alphafindimconv} exhibits a phase transition
according to the value of $p$. In the regime
$p<(\alpha+1)/(2\alpha)$, the preceding conditional argument also gives a
functional limit theorem.

\begin{theorem}\label{stable-functional-limit}
Suppose that $1<\alpha<2$ and
$p<(\alpha+1)/(2\alpha)$. Then
\[
 \left(
 \frac{\sum_{k=1}^{\lfloor nt\rfloor}\mathcal{Z}_k
 -(\lfloor nt\rfloor/n)b_n}{a_n}
 \right)_{t\geq0}
 \Rightarrow
 \big(\mathcal{S}_\alpha(t)\big)_{t\geq0}
\]
in $D([0,\infty))$ endowed with the $J_1$ topology, where
\[
 \mathcal{S}_\alpha(t)
 :=\mathcal{S}_\alpha^+(t/2)+\mathcal{S}_\alpha^-(t/2)
\]
is the stable L\'evy process with characteristic exponent $\Psi$. Since
$b_n=n\lambda$, the deterministic centering above may equivalently be
replaced by $tb_n$.
\end{theorem}

\begin{proof}
Set
\[
 Y_t^{(n)}:=
 \frac{\sum_{k=1}^{\lfloor nt\rfloor}\mathcal{Z}_k
 -(\lfloor nt\rfloor/n)b_n}{a_n},\qquad t\geq0.
\]
Using \eqref{stable-decomposition}, we obtain the exact decomposition
\[
 Y_t^{(n)}=
 \frac{\rho}{a_n}S_{\lfloor nt\rfloor}
 +\frac{1}{2a_n}\sum_{k=1}^{\lfloor nt\rfloor}
 \widetilde Z_k^+(1+X_k)
 +\frac{1}{2a_n}\sum_{k=1}^{\lfloor nt\rfloor}
 \widetilde Z_k^-(1-X_k),
\]
where $\widetilde Z_k^\pm=Z_k^\pm-\rho^\pm$. The almost sure ERW estimates
in \cite[Theorems 3.2, 3.5 and 3.7]{B18} and Potter bounds
\cite[Theorem 1.5.6]{BGT} imply that, for every
$T>0$,
\[
 \sup_{0\leq t\leq T}
 \frac{|S_{\lfloor nt\rfloor}|}{a_n}
 \xrightarrow{\PP}0,
\]
because $p<(\alpha+1)/(2\alpha)$. Thus the first term converges to zero
uniformly on compact time intervals in probability.

We treat the second term; the third one is analogous. Let
\[
 J_t^{(n)}:=\frac{1}{a_n}\sum_{k=1}^{\lfloor nt\rfloor}
 \widetilde Z_k^+(1+X_k),\qquad t\geq0.
\]
For completeness, the usual $J_1$ tightness conditions on $[0,T]$ are
\begin{equation}\label{tightG1}
 \lim_{r\uparrow\infty}\limsup_{n\to\infty}
 \PP\!\left(\|J^{(n)}\|_T>r\mid\mathcal{G}_\infty\right)=0
 \qquad\text{a.s.}
\end{equation}
and, for every $\epsilon>0$,
\begin{equation}\label{tightG2}
 \lim_{\delta\downarrow0}\limsup_{n\to\infty}
 \PP\!\left(w_T'(J^{(n)},\delta)>\epsilon
 \mid\mathcal{G}_\infty\right)=0
 \qquad\text{a.s.},
\end{equation}
where $\|\cdot\|_T$ is the uniform norm and
\[
 w_T'(x,\delta):=\inf_{\{t_i\}}
 \max_i\sup_{s,t\in[t_{i-1},t_i)}|x(t)-x(s)|,
\]
the infimum being taken over partitions
$0=t_0<t_1<\cdots<t_v=T$ satisfying
$\min_i(t_i-t_{i-1})>\delta$; see
\cite[Theorem 13.2]{Billingsley}.

We prove the stronger conditional functional convergence. Define
\[
 \widetilde A_n^+(u):=\frac{1}{a_n}
 \sum_{j=1}^{\lfloor nu\rfloor}\widetilde Z_j^+,
 \qquad u\geq0.
\]
Since $\widetilde Z_1^+$ belongs to the domain of attraction of
$S_\alpha^+$, the stable functional limit theorem gives
\[
 \widetilde A_n^+\Rightarrow\mathcal{S}_\alpha^+
\]
in the $J_1$ topology. Conditionally on $\mathcal{G}_\infty$, the variables
selected at the times for which $X_k=1$ remain i.i.d. with the distribution
of $\widetilde Z_1^+$. Consequently,
\[
 \big(J_t^{(n)}\big)_{t\geq0}
 \overset{d\mid\mathcal{G}_\infty}{=}
 \left(2\widetilde A_n^+\!\left(
 \frac{N_{\lfloor nt\rfloor}^+}{n}\right)\right)_{t\geq0}.
\]
By \eqref{clock-convergence} and the random time-change theorem
\cite[Section 13.2]{W02}, for almost every realization of the ERW,
\[
 \big(J_t^{(n)}\big)_{t\geq0}
 \Rightarrow
 \big(2\mathcal{S}_\alpha^+(t/2)\big)_{t\geq0}
\]
conditionally on $\mathcal{G}_\infty$. In particular,
\eqref{tightG1} and \eqref{tightG2} hold. The same argument gives
\[
 \left(\frac{1}{a_n}\sum_{k=1}^{\lfloor nt\rfloor}
 \widetilde Z_k^-(1-X_k)\right)_{t\geq0}
 \Rightarrow
 \big(2\mathcal{S}_\alpha^-(t/2)\big)_{t\geq0}
\]
conditionally on $\mathcal{G}_\infty$. The two limiting processes are
independent. Combining these convergences with the negligibility of the ERW
term completes the proof.
\end{proof}

\medskip

For $1<\alpha<2$, functional versions of the remaining cases of
Theorem \ref{alphafindimconv} follow from
\eqref{conditional-functional-limit} and the superdiffusive
functional limit theorem for the ERW; see
\cite[Theorem~3]{BB16}.
If $p>(\alpha+1)/(2\alpha)$, normalization by $n^{2p-1}$
yields the limit $(\rho t^{2p-1}L_{q,0})_{t\geq0}$:
by Potter bounds, the centered heavy-tailed contribution
vanishes uniformly on compact time intervals in probability.

At $p=(\alpha+1)/(2\alpha)$, the functional counterparts hold
with the same qualifications concerning $L$ and $\rho$ as in
Theorem \ref{alphafindimconv}. In particular, if
$L(n)\to\ell\in(0,\infty)$, normalization by $a_n$ yields
\[
 \left(
 \mathcal{S}_\alpha(t)+\frac{\rho}{\ell}t^{1/\alpha}L_{q,0}
 \right)_{t\geq0},
\]
with independent components. The required joint convergence
follows from \eqref{conditional-functional-limit}.
All these functional limits use the linear time scaling
$\lfloor nt\rfloor$ and the $J_1$ topology.


\medskip

\noindent {\bf Acknowledgments:} Glauco Valle would like to thank UFABC, where this project was started during a visit to Cristian Coletti. 

\bigskip



\end{document}